\documentclass[]{amsart}
\usepackage[left=2.7cm,right=2.7cm,top=3.5cm,bottom=3cm]{geometry}
\usepackage[english]{babel}
\usepackage{pifont}
\usepackage[latin1]{inputenc}
\usepackage[notref,notcite]{}
\usepackage{amsmath}
\usepackage{amsfonts}
\usepackage{amsthm}
\usepackage{amscd}
\usepackage{upgreek}
\usepackage{amssymb}
\usepackage[all]{xy}
\usepackage{graphicx}
\usepackage[usenames,dvipsnames]{color}
\usepackage{mathrsfs}
\usepackage{caption}
\usepackage{tikz-cd} 
\usepackage{braket}
\usepackage{mathrsfs}
\usepackage[colorlinks = true,
linkcolor = black,
urlcolor = black,
bookmarksopen = true, 
backref = page]{hyperref}

\theoremstyle{plain}
\newtheorem{thm}{Theorem}[section]

\newtheorem{lem}[thm]{Lemma}
\newtheorem{prop}[thm]{Proposition}

\theoremstyle{definition}

\newtheorem{rmk}[thm]{Remark}

\renewcommand{\ge}{\geqslant}
\renewcommand{\le}{\leslant}

\newcommand{\field}[1]{\mathbb{#1}}
\newcommand{\Q}{\field{Q}}
\newcommand{\C}{\field{C}}

\newcommand{\R}{\field{R}}

\newcommand{\Z}{\field{Z}}
\newcommand{\A}{\field{A}}
\newcommand{\F}{\field{F}}
\newcommand{\p}{\field{P}}
\newcommand{\G}{\field{G}}

\newcommand{\X}{\field{X}}

\newcommand{\nr}{nr}

\DeclareMathOperator{\PGL}{PGL}
\DeclareMathOperator{\End}{End}
\DeclareMathOperator{\Aut}{Aut}

\DeclareMathOperator{\Lie}{Lie}

\DeclareMathOperator{\val}{val}

\DeclareMathOperator{\Gal}{Gal}

\DeclareMathOperator{\GL}{GL}
\DeclareMathOperator{\SL}{SL}
\DeclareMathOperator{\Hom}{Hom}

\DeclareMathOperator{\Spec}{Spec}

\DeclareMathOperator{\Tr}{Tr}

\DeclareMathOperator{\Div}{Div}

\DeclareMathOperator{\pic}{Pic}
\DeclareMathOperator{\Spin}{Spin}

\DeclareMathOperator{\Ram}{Ram}
\DeclareMathOperator{\Gr}{Gr}

\DeclareMathOperator{\CM}{\scal^{\text{CM}}}

\DeclareMathOperator{\rec}{rec}

\DeclareMathOperator{\Stab}{Stab}
\DeclareMathOperator{\SO}{SO}

\DeclareMathOperator{\Cl}{Cl}
\DeclareMathOperator{\gen}{gen}
\DeclareMathOperator{\spn}{spn}
\DeclareMathOperator{\shi}{Shi}

\DeclareMathOperator{\picard}{Pic}

\newcommand{\cala}{\mathscr A}
\newcommand{\calu}{\mathcal U}

\newcommand{\calc}{\mathcal C}

\newcommand{\calf}{\mathcal F}
\newcommand{\calg}{\mathscr G}

\newcommand{\call}{\mathscr L}

\newcommand{\calo}{\mathscr O}

\newcommand{\cals}{\mathscr S}
\newcommand{\scal}{\mathcal S}
\newcommand{\calt}{\mathscr T}

\newcommand{\calw}{\mathcal W}

\newcommand{\gota}{\mathfrak a}

\newcommand{\goto}{\mathfrak o}
\newcommand{\gotm}{\mathfrak m}
\newcommand{\gotn}{\mathfrak n}
\newcommand{\gotp}{\mathfrak p}

\renewcommand{\ge}{\geqslant}
\renewcommand{\le}{\leqslant}

\newcommand{\ndiv}{\hspace{-4pt}\not|\hspace{2pt}}

\newcommand{\interior}[1]{%
	{\kern0pt#1}^{\mathrm{o}}%
}
\newtheorem{innercustomgeneric}{\customgenericname}
\providecommand{\customgenericname}{}
\newcommand{\newcustomtheorem}[2]{%
	\newenvironment{#1}[1]
	{%
		\renewcommand\customgenericname{#2}%
		\renewcommand\theinnercustomgeneric{##1}%
		\innercustomgeneric
	}
	{\endinnercustomgeneric}
}

\newcustomtheorem{customthm}{Theorem}

\numberwithin{equation}{section}

\title{Equidistribution of CM points  on  Shimura Curves and ternary theta series}
\author{Francesco Maria Saettone}
\address{Department of Mathematics, Weizmann Institute of Science, Israel}
\email{francesco.saettone@weizmann.ac.il}

\begin{document}
	
	\maketitle
	
	\begin{abstract}
		We prove an equidistribution statement for the reduction of Galois orbits of CM points on the special fiber of a Shimura curve over a totally real field, considering both the split and the ramified case. The main novelty of the ramified case consists in the use of the moduli interpretation of the Cerednik--Drinfeld uniformisation. Our result is achieved by associating to the reduction of CM points certain Hilbert modular forms of weight $3/2$ and by analyzing their Fourier coefficients. Moreover, we also deduce the Shimura curves case of the integral version of the Andr\'e--Oort conjecture.
	\end{abstract}
	
	\tableofcontents
	
	\section{Introduction}
	
In \cite{jk} Jetchev and Kane proved an equidistribution result for the reduction of Galois orbits of Heegner points in a modular curve with both the conductor and the discriminant going to infinity. In this work we generalize their result to quaternionic Shimura curves over a totally real number field $F$, where we also allow the quaternion algebra both to be split and to ramify at the prime at which the reduction takes place. 
\\

Since the beginning of this century, the equidistribution of the reduction of Galois orbits of CM points has been a fruitful area of research. In \cite{michel}, Michel gave a subconvexity bound for certain Rankin--Selberg $L$-functions and used this bound to prove an equidistribution property for Galois orbits of supersingular elliptic curves. His result was recently generalized to a simultaneous equidistribution in \cite{aka}, exploiting dynamics and ergodic theory. On the other hand, Cornut and Vatsal in \cite{cv} proved that the reductions of CM points (modulo a non-split prime in a CM extension of $F$) of a Shimura curve are equidistributed in the supersingular locus. This was obtained by the mean of Ratner's theorem, with the noticeable application consisting of the (generalized) Mazur's conjecture on the non-triviality of higher Heegner points and on the non-vanishing of central values of automorphic $L$-functions. In \cite{fms}, I proved a companion equidistribution on the special fiber of a Shimura curve attached to some ramified primes. Moreover, in \cite{xz} a variant of these results for (Hecke orbits of) PEL Shimura varieties is treated. Lastly, Jetchev and Kane in \cite{jk} partially generalized \cite{cv} for modular curves. 

For a overview on analogous recent $p$-adic equidistributions, we refer to \cite{dis} and its citation orbit. In the rest of this introduction, we concisely state our main result, the idea of its proof, and a diophantine application.
\\

Let $B$ be an indefinite quaternion algebra over $F$. Let also $K$ be a CM extension of $F$. In this work we deal with both the following situations:
\begin{enumerate}
	\item $B$ is ramified at $v$ and $v$ either ramifies or is inert in $K$;
	
	\item $B$ is unramified at $v$ and $v$ is inert in $K$.
\end{enumerate}

For the sake of clarity, we underline that in the setting of Shimura curves CM points are sometimes called Heegner points, in particular in the case of modular curves. We follow the terminology of CM point to differentiate between the points on the Shimura curve and the corresponding points on an elliptic curve, which we do not discuss.
\\

Let $\{s_1,...,s_h\}$ denote either the supersingular or superspecial locus\footnote{Note that these two sets do not have the same cardinality.} of the special fiber of a Shimura curve. The moduli interpretations of its local integral models at $v$ allow us to consider the endomorphism ring $\End(s_i)$ of the supersingular point, and to define $w(s_i)$ as the number of units modulo $\{\pm1\}$ of $\End(s_i)$. In the superspecial case, the analogous $w(s_i)$ is defined through the construction of Section \ref{Grosslattice}. 
\\Let $\diamond\in\{\text{ss},\text{ssp}\}$, and define $\mu^{\diamond}$ be the normalized counting measure on the  supersingular (respectively, superspecial) locus $\{s_1,...,s_h\}$ of special fiber of the Shimura curve given by
\[
\mu^{\diamond}(s_i)=\dfrac{w(s_i)^{-1}}{\sum_{j=1} ^{h}w(s_j)^{-1}}\;.
\]
Moreover, in the case of $B_v$ ramified, let also $\mu_{\text{in}}^{\text{ssp}}$ be the normalized counting measure on  the set of irreducible component $\{c_1,...,c_k\}$, given by
\[
\mu_{\text{in}}^{\text{ssp}}(c_i)=\dfrac{w(c_i)^{-1}}{\sum_{j=1}^k w(c_j)^{-1}}
\]
where  $w(c_i)=\#\End(c_i)/2$ is the $\mathit{weight}$ of $c_i$. Here by $\End(c_i)$ we mean the endomorphisms of the formal group attached to the reduction of a CM point modulo an inert prime, which lands on the connected component $c_i$.

The main result of our work goes as follows.

\begin{customthm}{A}\label{theoremA}
	 The reductions at $v$ of the Galois orbits of CM points are equidistributed for the discriminants and the conductors varying, i.e., for their absolute norms going to infinity:
	 \begin{itemize}
	 	\item in the supersingular locus of the Shimura curve with respect to $\mu^{\text{ss}}$ for $B_v\simeq M_2(F_v)$  and $v$ inert or ramified in $K$;
	 	
	 	\item  in the superspecial locus of the Shimura curve with respect to $\mu^{\text{ssp}}$ for $B_v$ ramified and $v$ ramified in $K$;
	 	
	 	\item in the smooth locus of the Shimura curve with respect to $\mu_{\text{in}}^{\text{ssp}}$ for $B_v$ ramified and $v$ inert in $K$.
	 \end{itemize}
	 
\end{customthm}

This theorem gives a generalization of \cite{cv}, \cite{fms} and, of course, of \cite{jk}. Stricto sensu, this generalization for \cite{cv} and \cite{fms} is only partial: in fact, although we do not fix the discriminant, we  only manage to reduce by a single prime, losing the simultaneous reduction of the two aforementioned works.
%\begin{rmk}
%	The restriction on the inertia of $2$ in $F$ exclusively concerns an auxiliary result on quadratic forms, namely Lemma \ref{jon} (and therefore Lemma \ref{disc}, necessary for Lemma \ref{jon}), and it could be lifted at the cost of some (unpleasant) technical computations.
%\end{rmk}

	\begin{figure}
		\centering
		\includegraphics[width=0.5\linewidth]{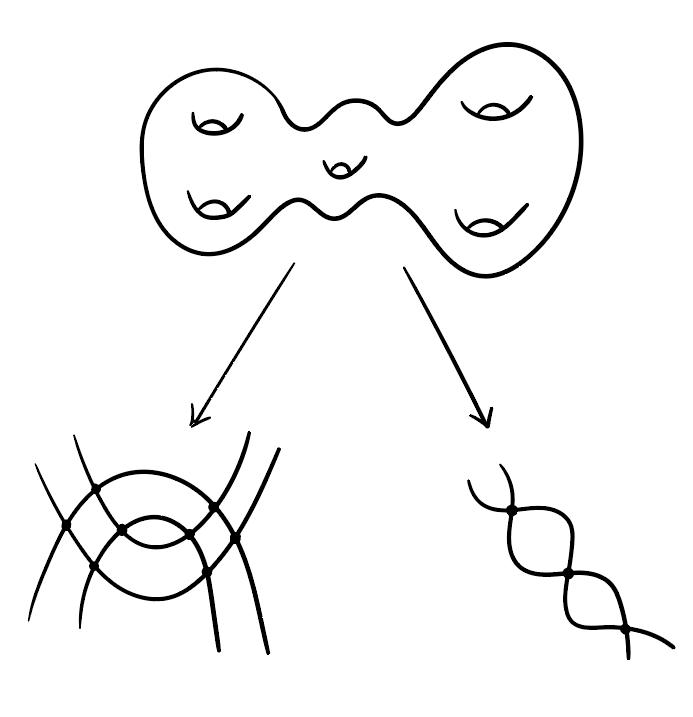}
		\caption{The illustration (drawn with the help of Francesco Beccuti) represents the Shimura curve over $\Q$ of discriminant $77$ and maximal level structure, whose CM points reduced modulo $7$ land in the superspecial locus (on the left) and to the supersingular locus for a prime $p\neq 7,11$. Note that on the right hand side the two curves should intersect $6$ times, since the special fiber is of genus $5$.}
		%\label{fig:tereqimage}
	\end{figure}

Let us briefly sketch the main objects and techniques we make use of. 
\\We begin by describing the special points of the indefinite and definite Shimura curves, both as adelic double quotients as in \cite{sz} and in terms of their moduli interpretations. In order to give these two descriptions also for their reduction in the special fiber at $v$, we turn to study the local integral models of the Shimura curve at $v$, for $B_v$ unramified and ramified respectively. This changes dramatically their geometry and consequently the arithmetic of their moduli problems. In particular, in the ramified case we exploit the Cerednik--Drinfeld uniformisation as in \cite{bz} and the related moduli intepretation of the Drinfeld plane in terms of formal $\calo_{B_v}$-modules. This is the main geometric innovation in the equidistribution context. On the other hand, if $B_v\simeq M_2(F_v)$, only a local uniformisation in terms on Lubin--Tate spaces is available, and we essentially use the moduli interpretation of \cite{car}. All of this eventually allows us to state a correspondence between the set of CM points $x$ of  fixed discriminant and conductor reducing to a supersingular or superspecial point $s$ and the set of (conjugacy classes of) optimal embeddings $\End(x)$ into $\End(s)$. This correspondence has a ``numerical" incarnation in the formula of Lemma \ref{lem2.7}.

Subsequently, we discuss Hilbert modular forms of half-integral weight and their automorphic counterparts, so to exploit an adelic setting. To a supersingular or superspecial point $s$, we associate an Eichler order $R_s$ and then the so called Gross lattice, that is, a ternary quadratic form $Q_s$ attached to $s$. Since at least Jacobi, quadratic forms and automorphic forms are allied subjects, and in fact $Q_s$  gives rise to a theta series of weight $3/2$. Such theta series were introduced in Gross' seminal paper \cite{gr}, and they formed already a key idea in \cite{elkies} and similarly in \cite{jk} where the authors exploited a subconvexity bound on their Fourier coefficients. %by Duke and Iwaniec.
Now, the set of the optimal embeddings previously introduced is in correspondence with primitive representations of the discriminant of $x$ by (the adelization of) $Q_s$. Next step, which was the main technical novelty of \cite{jk}, consists in the analysis of the Whittaker--Fourier coefficients of the (automorphic) theta series attached to the genus and the spinor genus of $Q_s$. In particular, we prove that, under certain conditions, the genus and spinor genus mass, as defined in (\ref{genmass}), are equal. To prove this, we make use of some class number formula and of Vatsal's equidistribution in Lemma \ref{vat}. Moreover, we also need to adapt to totally real fields some properties of those theta series. To do so, we decided to operate in an automorphic setting. We thus exploit the beautiful results of Gelbart--Piatetsky-Shapiro \cite{gps1} and \cite{gps}: these works develop an automorphic, representation theoretic Shimura correspondence, which can be viewed as a first case of the theta-correspondence. 
	
The last crucial ingredient for Theorem A follows from the subconvexity bound of \cite{bh} on the Whittaker--Fourier coefficients of cuspidal automorphic forms of weight $3/2$. Due to Brauer--Siegel's lower bound for class numbers, our results are ineffective. However, this issue might be solved by assuming the Generalized Riemann Hypothesis  or by the (weaker assumption of the) non-existence of Landau--Siegel zeros, at the price of obtaining a conditional result. For the sake of simplicity, consider now the case $F=\Q$. Let $D$ be the discriminant of an imaginary quadratic extension of $\Q$ with class number $h(D)$. Then by one of the two assumptions above Hecke showed there is an effective constant $c>0$ such that $h(D)>c\sqrt{|D|}(\log|D|)^{-1}$, which gives the desired effective, albeit conditional, lower bound.
\\

Lastly, in Section \ref{sectionAO}, as a complement of our equidistribution result, we deduce the integral version of the Andr\'e--Oort conjecture in arithmetic pencil introduced in \cite{aoexp} in the case of Shimura curves over $\Q$. While the classical Andr\'e--Oort conjecture concerns the Zariski topology of Shimura varieties, this variant considers their integral models, which are defined over $\Spec(\calo_E)$, where $E$ is their reflex field, hence the pencil over the arithmetic base $\Spec(\calo_E)$.

The strategy closely follows the lines of \cite{ao}, where the case of the modular curve $Y(1)$ over $\Z$ is established. We underline that our framework encompasses both the modular and the Shimura case.
\\

Let us conclude this preamble with a few  words on future directions. There has recently been a tremendous activity on the Mixing Conjecture of Michel--Venkatesh, especially via the joining theorems of Einsiedler and Lindenstrauss, as we invite the interested reader to go through  \cite{akabourbaki} for an introduction to these powerful ideas. In particular, the work of Khayutin \cite{kha} and Blomer--Brumley--Khayutin \cite{brumley} proves, under some technical assumptions\footnote{Varying from the GRH and the Ramanujan conjecture to some mild, albeit significant, splitting condition in the CM extension.}, the Mixing Conjecture, with the noticeable consequence of the equidistribution of Galois orbits of CM points on a product of (indefinite) Shimura curves over $\Q$ which yields, which implies the Andr\'e--Oort conjecture for such a product, as showed in \cite[Remark 1, p.3660]{swzequi} and explained in \cite[Section 2.2]{brumley}.  It would thus be very interesting to extend Theorem \ref{theoremA} to a product of Shimura curves and explore diophantine applications as the aforementioned integral version of the Andr\'e--Oort conjecture.
	
\subsection{Notation and Conventions}

	We list some notations used throughout this work:
	\begin{itemize}
	
		\item we denote by $\mathbb{H}$ the classical Hamilton's quaternions;
		\item we denote by $\A$ the ring of adeles over $\Q$, and $\A_f$ the finite adeles;
		\item for a totally real number field $F$, let $F^+$ denote its totally positive elements;
		\item for a quaternion algebra $B$ over $F$, we denote by $\Ram(B)$, the set of places of $F$ where $B$ ramifies; the same symbol, with $f$ or $\infty$ subscript, denotes the restriction to finite or archimedean places respectively;
		\item for a set $S$, we denote by $\mathbf{1}_S$ the characteristic function of $S$;
		\item for a local field $F_v$, we denote by $\breve{F}_v$ the completion of the maximal unramified extension;
		\item let $G^{\text{ab}}_K$ denote the abelianization of the absolute Galois group of $K$;
		\item for a set $X$, we denote by $\calc(X,\C)$ the set of continuous functions on it;
		\item  for an $F$-group scheme $G$, we denote its automorphic quotient by $[G]:=G(\A)/G(F)$; 
		\item we denote by $N$ the absolute norm of a number field, while $\nr$ denotes the reduced norm of a quaternion algebra;
		\item for an algebra $A$, we write $\pic(A)$ instead of $\pic(\Spec A)$;
		\item for any two functions $f,g$, we use the notations $f\ll g$ and $f=O(g)$ interchangeably.
	\end{itemize}

\subsubsection*{Acknowledgments}

It is a pleasure to thank Professor Daniel Disegni for his guidance with my doctoral thesis, which this article is part of. I am also grateful to Zev Rosengarten for useful conversations and correspondence. This work was supported by ISF grant 1963/20 and BSF grant 201825, by ERC (SharpOS, 101087910), and by the ISF grant  2067/23.

\section{Shimura Curves}

\subsection{Quaternion Algebras and Ramification}

For a totally real number field $F$ of degree $d$ over $\Q$, let $B$ be a quaternion algebra  over $F$.
In this section, after imposing some conditions on the ramification of the quaternion algebra $B$ at some places of $F$, we associate two curves to $B$, according to the fact that $B$ satisfies one of the following conditions:
\begin{itemize}
	\item  there is a unique real place $v$ of $F$ such that $B_v=B\otimes F_v\simeq M_2(F_v)$, i.e., $B$ is $\mathit{indefinite}$;
	
	\item for every real place $B_v$ is non-split, i.e., $B$ is $\mathit{definite}$.
\end{itemize}

\subsubsection{Ramification Setting}

From now on, let $K$ denote a CM quadratic extension of $F$, and suppose that $B$ is split by $K$, i.e., $K_v$ is a field for every finite place of $F$ where $B$ ramifies. Fix an embedding $\rho\colon K\hookrightarrow B$ over $F$. Let also $v$ be a finite place of $F$. In what follows we consider the following two situations:
\begin{enumerate}
	\item $B$ is ramified at $v$ and $v$ ramifies in $K$;
	
	\item $B$ is unramified at $v$ and $v$ is inert in $K$.
\end{enumerate}

Indeed, this does not rule all the possible cases out. For instance, in $(1)$, the prime $v$ could also be inert. The geometric interpretation of this case is  described in \cite[Prop.2.13]{fms}. At any rate, at the level of the equidistribution both possibilities in $(1)$ would proceed in parallel, so we decided to focus exclusively on $v$ ramifying in $K$. 

\subsubsection{Indefinite Case}

We first consider $B$ indefinite at one archimedean place of $F$, which we denote by $\tau_1\colon F\hookrightarrow \R$. Then we can fix an isomorphism $B\otimes\R\simeq M_2(\R)\oplus\mathbb{H}^{d-1}$.

Let $G$ be the reductive group over $\Q$ whose functor of points sends a commutative $\Q$-algebra $A$ to
\begin{equation}\label{gr}
G(A)=(B\otimes_\Q A)^\times.
\end{equation}
This means that $G=\text{Res}_{F/\Q}B^\times$ and there is a real embedding $h_0\colon \text{Res}_{\C/\R}\G_m\rightarrow G_\R$ with trivial coordinates at $\tau_i$ for $i\ge2$.
\\ The isomorphism $B\otimes\R\simeq M_2(\R)\oplus\mathbb{H}^{d-1}$ induces a map $B^\times\rightarrow \GL_2(\R)$ which gives an action of $B^\times$ on the conjugacy class of $h_0$, which is isomorphic to $\mathscr{H}=\C-\R$, by M\"obius transformations. 

For any open subgroup $U$ of $G(\A_f)$ which is compact modulo $\widehat{F}^\times$, we consider
\begin{equation}\label{anShim}
\scal_U^{\text{an}}:=G(\Q)\backslash \mathscr{H}\times G(\A_f)/U\cup \{\text{cusps}\}
\end{equation}
whose canonical model is the Shimura curve $\scal_U$ over $F$. It is a proper and smooth curve over its reflex field $F$. Note that the set $\{\text{cusps}\}$ is non-empty if and only if $B=M_2(\Q)$. 
If $F=\Q$, then $\scal_U$ is a coarse moduli space of QM abelian surfaces\footnote{A.k.a. false elliptic curves, i.e.,  $A$ is an abelian scheme over $R$ of relative dimension $2$.}, see \cite{buzzard}. 
\\On the other hand, if $d>1$, Shimura curves do not have a natural moduli interpretation. Despite so, by the work of Carayol \cite{car}, we have that $\scal_U$ has a finite map to another Shimura curve $\scal_{U'}$ which, if $U'$ is small enough, is a moduli space of QM abelian varieties with a $U'$-level structure and a polarization both compatible with the quaternionic multiplication (see  \cite[Prop.1.1.5]{sz2}). Since $\scal'_{U'_v}$ is a fine moduli scheme for such abelian schemes over schemes with level structures, there is a universal object of its moduli problem, called $\mathit{universal}$ abelian surface and denoted by $\mathscr{A}\rightarrow \scal'_{U'_v}$. For each geometric point $x=\Spec(E)$ of $\mathscr{A}$, the fiber $\mathscr{A}_x$ is a QM abelian surface with the above structure, defined over $E$.

The moduli interpretation yields the remarkable consequence that the Shimura curve has a proper\footnote{In the modular curves case one needs to add finitely many cusps for properness.} regular integral model $\cals_U$ over $\calo_F$.

\subsubsection{Definite Case}\label{defcase}

Let us now deal with a definite quaternion algebra. For the rest of this work fix a finite prime $\ell$ of $F$ such that $B_\ell\simeq M_2(F_\ell)$. Consider an Eichler order $R$ of level $\gotn$. An $\mathit{orientation}$ on $R$ consists of a morphism $\goto_\star\colon R\otimes k_\ell\rightarrow k_{\star}$, for $\star\in\{\ell,\ell^2\}$, where $k_\star=k_\ell$ if $\ell|\gotn$ and $k_\star=k_{\ell^2}$ otherwise. We refer to \cite[Section 1.1]{bd} for an exhaustive description of such orientations. 

We denote by $\Cl(B)$ the set of all conjugacy classes of oriented Eichler orders of level $\gotn$ in $B$, and we recall that it can be viewed adelically via the following bijection
\begin{equation}\label{shimuraset}
\Cl(B)\simeq G(\Q)\backslash G(\A_f)/\widehat{R}^\times,
\end{equation}
where $\widehat{R}^\times$ denotes the adelization of $R$ and $G$ is the reductive group defined in (\ref{gr}) adapted to the definite setting. By strong approximation in $G(\A_f)$, it follows that $\Cl(B)$ is finite.
\\

Following \cite[p.131]{gr}, we now attach to $B$ an algebraic curve. Let $\mathscr{P}$ be the conic curve over $\Q$ whose functor of points send a commutative $\Q$-algebra $A$ to
\[
\mathscr{P}(A)=\{x\in B\otimes A : x\neq 0,\; \nr(x)=tr(x)=0 \}/A^\times.
\]
We define the $\mathit{Gross}$ $\mathit{curve}$\footnote{Also known as $\mathit{definite}$ Shimura curve, as in \cite[p.420]{bd}.} of level $R$ as
\[
X_R=G(\Q)\backslash \mathscr{P}\times G(\A_f)/\widehat{R}^\times,
\]
where $G(\Q)=\Aut(\mathscr{P})$ acts by conjugation on $\mathscr{P}$. Given representatives $(g_i)_{i=1}^r$ of $\Cl(B)$ via (\ref{shimuraset}), we define $\Gamma_i:=g_i\widehat{R}g_i^{-1}\cap G(\Q)$, which is a finite subgroup of $G(\Q)$. We thus obtain a collection of genus zero conic curves\footnote{Which are in general not isomorphic to $\p^1$.} $Y_i=\Gamma_i\backslash\mathscr{P}$ defined over $\Q$, and we write the Gross curve as
\[
X_R=\bigsqcup_{i=1}^r Y_i.
\]

The set of $K$-rational points of the conic curve $\mathscr{P}$ is identified with $\Hom(K,B)$, as explained in \cite[p.131]{gr}. By this identification, we obtain that
\[
X_R(K)=\{(f,[R]) : f\in\Hom(K,B),\;[R]\in\Cl(B)\}.
\]
Lastly, since the connected components $Y_i$ have genus zero, the Picard group $\picard(X_R)$ is a free $\calo_F$-module with a basis indexed by $\Cl(B)$.

\subsection{Special Points and Quaternion Algebras}

\subsubsection{Indefinite Case}

For any $F\otimes\A_f$-embedding $\tau\colon K\otimes\A_f\hookrightarrow B\otimes\A_f$, the group $\tau(K^\times)$ acts on the Shimura curve $\scal_U$. We thus define the scheme of $\mathit{CM}$ $\mathit{points}$ by $(K,\tau)$ as the fixed-point (affine) subscheme $\scal_U^{\tau(K^\times)}$. 
\\In other words, a point $z$ of $\scal_U$ is a CM point by $K$ if it can be represented by $(z_0,g)\in \mathscr{H}\times G(\A_f)$ via (\ref{anShim}), where $z_0$ is the unique point fixed by $K^\times$. By the work of Shimura, it is a finite subscheme of $\scal_U$ defined over $K^{\text{ab}}$. By taking the union of $\scal_U^{\tau(K^\times)}$ over all pairs $(K,\tau)$, we obtain the CM ind-subscheme of $\scal_U^{\text{CM}}$. The absolute Galois group of $K$, which we denote by $G_K$, acts on $\scal_U^{\text{CM}}$ via
\[
\sigma.(x_0,g)=(x_0,\rec_K(\sigma)g)
\]
where $\rec_K$ is Artin's reciprocity map. If we consider CM points of conductor $c$, this action factors through $\text{Gal}(H[c],K)$, where $H[c]$ is the ring class field of $K$ of conductor $c$.
\\

Consider now an order $R$ of $B$ of type\footnote{I.e., an order of discriminant $\gotn$ which contains $\rho(\calo_K)$.} ($\gotn$, $K$) as constructed in \cite[Section 1.5.1]{sz2} and the corresponding  Shimura curve of level $\widehat{F}^\times\widehat{R}^\times$. Then for $z$ a CM point by $K$ we consider 
\[
\End(z):=g\widehat{R}^\times g^{-1}\cap \rho(K)
\]
which is an order in $K=\rho(K)$ independent of the choice of $g\in G(\A_f)$. The $\mathit{conductor}$ of $z$ is defined as the unique $\calo_F$-ideal such that
\[
\End(z)=\calo_F + c\calo_K.
\]
Moreover, the discriminant of $\End(z)$ is of the form $Dc^2$, where $D$ is the discriminant of $K$ relative to $F$ and $c$ the conductor.
%c^2 largest square factor of D such that D/C^2\equiv 0,1 \bmod 4.
\\

We say that a CM point $z$ corresponding to $(z_0,g)$ has an $\mathit{orientation}$ if, for a finite prime $v$ coprime with $c$, the morphism $g^{-1}\rho g$ is $R_v^\times$-conjugated to $\rho$ in $\Hom(\calo_{K_v},R_v)/R_v^\times$.

\begin{lem}\label{numberGalorbits}
	Let $c$ be coprime with $\gotn$. Then there are $2^{\omega(\gotn)}$ Galois orbits of CM points of conductor $c$, where $\omega(\gotn)$ is the number of prime factors of $\gotn$. Moreover, every such orbit has cardinality equal to $\#\picard(\calo_c)$.
\end{lem}
\begin{proof}
	Let us recall that the Shimura curve of level $\widehat{R}^\times$ has an action by the Atkin--Lehner group
	\[
	\calw=\{b\in G(\A_f) : b^{-1}\widehat{R}^\times b =\widehat{R}^\times \}/\widehat{R}^\times
	\]
	which has $2^s$ elements, where $s$ is the number of prime factors of $\gotn$. This yields an action of $\calw$ on $\scal_U^{\text{CM}}$ which preserves the conductor.
	By the work of Gross \cite{gr}, the action of $\picard(\calo_c)\times \calw$
	on the set of CM points of conductor $c$ is free and transitive and the $\picard(\calo_c)$-orbits of CM points correspond to sets of oriented CM points. This implies that there are $2^{\omega(\gotn)}\cdot \#\picard(\calo_c)$ such points of conductor $c$, and so we conclude.
\end{proof}

Next description characterizes CM points in terms of an adelic double quotient. Let $T$ be the $\Q$-rational torus in $G$.

\begin{lem}\label{adelicCM}
	The set of CM points in $\scal_U$ is in bijection with
\[
		 T(\Q)\backslash G(\A_f)/U.
\]	
\end{lem}
\begin{proof}
	See \cite[5.2.5.2.2]{sz}.
\end{proof}

\subsubsection{Definite Case}

Here we consider the definite case, i.e., special points living on some Gross curve, whose theory is considerably less convoluted than in the indefinite case.

For any order $\calo$ of $K$, we fix an orientation on it by choosing a morphism $\calo_\ell\rightarrow k_\star$ where $k_\star$ is as in Section \ref{defcase}. 

A $\mathit{Gross}$ point of conductor $c$ consists of a pair $(f,R)$ where $f\colon K\rightarrow R$ is an oriented optimal embedding, which means that $f(K)\cap R=f(\calo_c)$ where $\calo_c$ is taken up to conjugation by $G(\Q)$. Equivalently, Gross points are the image of 
$$\mathscr{P}(K)\times G(\A_f)/\widehat{R}^\times$$ in $X_R(K)$. This immediately implies the $K$-rationality of Gross points. Moreover, a Gross point $(f,R)$ represented by $(x,g)\in\mathscr{P}\times G(\A_f)$ has discriminant $D$ if and only if 
\[
f(K)\cap g\widehat{R}g^{-1}=f(\calo)
\]
is the image of the order of discriminant $D$. Thus, Gross points of discriminant $D$ correspond to equivalence classes modulo $R_i^\times$ of optimal embedding of $\calo\hookrightarrow R_i$ in each component $Y_i$. We denote the cardinality of the set of such equivalence classes by $h(\calo,R_i)$.  Let us denote by $\Gr(c)$ the set of Gross points of conductor $c$.

Set $T$ to be the $\Q$-rational torus in $G$. For any Gross point $(f,R)\in\Gr(c)$ the optimal embedding $f$ induces, by scalar extension, a map $T(\A)\rightarrow G(\A)$. This gives an action of the Picard group of $\calo_c$
\[
\picard(\calo_c)=T(\Q)\backslash T(\A)/\widehat{\calo}_c^\times 
\]
on $\Gr(c)$. By \cite[p.133]{gr}, this action is simply transitive.

Next result recalls how to describe Gross points as an adelic double quotient.

\begin{lem}\label{adelicgross}
	The set $\Gr$ of Gross points in $X_R$ is in bijection with 
    \[
    T(\Q)\backslash G'(\A_f)/\widehat{R}^\times.
    \]
\end{lem}
\begin{proof}
	See \cite[Lemma 2.2]{jk}.
\end{proof}

\subsection{Reductions of Integral Models}\label{reductionintegralmodels}

In this section we consider the indefinite quaternion algebra $B$ exclusively.

\subsubsection{Universal $v$-Divisible Group}\label{universal}

Let $A$ be a complete Noetherian local ring of residue characteristic $p$, and let $v$ be a non-archimedean place of $F$. A $v$-divisible\footnote{Also known as $\mathit{Barsotti}$--$\mathit{Tate}$.} group $\calg$ of height $h$ over $A$ is the colimit of a tower of affine, finite flat group schemes $(\calg_n)_{n\ge 1}$ over $A$ of order $p^{nh}$ such that $\calg_n=\calg_{n+1}[v^n]$. Taking the connected and \'etale parts\footnote{Defined just by the latter group schemes $\calg_n$.} of $\calg$, denoted by $\calg^\circ$ and $\calg^{\text{\'et}}$, there is a short exact sequence
\[
0\rightarrow \calg^\circ\rightarrow\calg\rightarrow\calg^{\text{\'et}}\rightarrow 0
\]
which splits if $A$ is a perfect field. Moreover, let us recall that a $\mathit{formal}$ group is a group object in the category of formal schemes. Then there is an equivalence of categories between connected $v$-divisible groups and formal groups.
\\

We now assume that the level structure decomposes in the form $U=U^v\cdot U_v$ where $U^v$ is sufficiently small and $U_v\simeq\calo_{B_v}^\times$, namely the level structure is maximal at $v$.

The curve $\cals_{U_v}$ carries a universal $v$-divisible $\calo_{B_v}$-module $\calg$ which comes from the $v$-divisible group $\mathscr{A}[v^\infty]$ by taking the pull-back of $\mathscr{A}$ via the finite map $\cals_{U_v}\rightarrow \cals'_{U'_v}$. To describe $\calg$, choose an auxiliary quadratic field $F'$ as in \cite[p.33]{sz2} which is split at $v$ and fix an isomorphism $f\colon \calo_{F'_v}\simeq\calo_{F_v}\oplus\calo_{F_v}$. We therefore define
\[
\calg:= f^{-1}(0,1)\mathscr{A}[v^\infty].
\]
Let $\calo_{\text{un}}$ be any unramified quadratic extension of $\calo_{F_v}$ contained in $\calo_{B_v}$. By \cite[Prop.1.2.4]{sz2}, we have that $\calg$ is a $\mathit{special}$ $v$-divisible $\calo_{B_v}$-module, i.e., the induced action of $\calo_{B_v}$ on $\Lie(\calg):=\Lie(\calg^\circ)$ makes $\Lie(\calg)$  a locally free sheaf over $\calo_{\cals}\otimes_{\calo_{F_v}}\calo_{\text{un}}$ of rank $1$. Moreover, $\calg$ is \'etale away from $v$.

\subsubsection{Unramified Case}

Assume that $B$ is unramified at $v$, and fix an isomorphism $j\colon\calo_{B_v}\simeq M_2(\calo_v)$. Let also be the level structure $U$ as in Section \ref{universal}. By \cite{car} (and by \cite[Chapter 13]{KM} if $F=\Q$), we have that $\cals_{U_v}$ has good reduction.

Since we have fixed $j$, every $\calo_{B_v}$-module $M$ can be uniquely decomposed via idempotents as
\[
M=M^1\oplus M^2:=\begin{bmatrix}
	1 & \\
	&
\end{bmatrix}M\oplus\begin{bmatrix}
   & \\
   & 1
\end{bmatrix} M
\]
where the $\calo_{F_v}$-modules $M^1$ and $M^2$ are isomorphic. Therefore we can write $\calg$ as
\[
\calg=\calg^1\oplus\calg^2
\]
with the summands are isomorphic as $v$-divisible $\calo_{F_v}$-modules.

Let $s$ be a geometric point of the special fiber of $\cals_U$ at $v$. Let also $\underline{F_v/\calo_{F_v}}$ denote the constant $v$-divisible group of height $1$ given by the colimit of $(\underline{v^{-n}\calo_{F_v}/\calo_{F_v}})$, which is \'etale. Since the $v$-divisible $\calo_{F_v}$-modules $\calg^i_s$ have height $1$ and dimension $2$, they are isomorphic to one of the following objects:
\begin{enumerate}
	\item the direct sum $X_1\oplus\underline{F_v/\calo_{F_v}}$, where $X_1$ is the unique\footnote{Therefore $X_1$  coincides with the geometric generic fiber of the Lubin-Tate formal group $\call\calt$, i.e., $\call\calt|_{\Spec(\overline{F}_v)}$.} formal $\calo_{F_v}$-module of height $1$, so that $s$ is called $\mathit{ordinary}$;
	
	\item to the unique formal $\calo_{F_v}$-module of height $2$ and dimension $1$, so that $s$ is called $\mathit{supersingular}$.
\end{enumerate}
In other words, $\calg_s^i$ are supersingular if $\calg_s^i=\calg^\circ$, while they are ordinary if the \'etale part $\calg^{\text{\'et}}$ is non-trivial.

Let $k$ be the algebraic closure of the residue field of $F$ at $v$. We denote by $\cals_{U,k}^{\text{ss}}$ the finite \'etale subscheme of supersingular points over $\Spec(k)$. We will also refer to it as the $\mathit{supersingular}$ locus. The complement of this finite set of points is naturally called the $\mathit{ordinary}$ locus.

Let us recall that supersingular points are mutually isogenous, and that, given such a point $s$, its endomorphism ring $\End(s)\otimes F$ is the quaternion algebra $B'$ obtained by switching invariants at $\tau$ and $v$ in $B$.

Geometrically, the irreducible components of special fiber $\cals_{U,k}$ are smooth connected curves\footnote{Never isomorphic to $\p^1$.} and they intersect each other transversally only in the supersingular locus, so that supersingular points corresponds to the singularities of the special fiber. We refer to \cite[9.4.4]{car} for more details.

Next Lemma gives an adelic description of the supersingular points.

\begin{lem}\label{adelicss}
	The set of supersingular points in $\cals_{U,k}$ is in bijection with
	\[
	G'(\Q)\backslash G'(\A_f)/ U',
	%\widehat{R'}^\times,
	\]
	%where $\widehat{R'}^\times$ is the idelization of an Eichler order of $B_v'$.
	where $U'=\calo_{B'_v}^\times\cdot U^v$.
\end{lem}
\begin{proof}
	See \cite[p.259]{sz}.
\end{proof}

\subsubsection{Ramified case}

Assume that $B$ is ramified at $v$ and that $U_v$ is maximal. In this setting, by the work of Drinfeld \cite{drin} we have that $\calg=\calg^\circ$, i.e., $\calg$ is a formal group. 

Let $B'$ be the quaternion algebra obtained by changing the invariant of $B$ at $v$ and $\tau$, i.e., $B'$ is definite and unramified at $v$. Consequently, we denote by $G'$ the reductive group whose functor of points is given by $G'(A)=(B\otimes_\Q A)^\times$, for a $\Q$-algebra $A$. We then fix an isomorphism $G'(\A_f)\simeq M_2(F_v) \cdot G(\A_f^v)$. 

Consider the formal scheme $\widehat{\Omega}$ over $\calo_{F_v}$ obtained by successive blow-ups of rational points on the special fiber over $k$ of $\p^1$, whose generic fiber is the rigid-analytic Drinfeld plane $\Omega$ over $F_v$ such that
\[
\Omega(\C_v)=\p^1(\C_v)-\p^1(F_v).
\]
We denote by $\widehat{\cals}_U$ the completion of $\cals_U$ along its special fiber at $v$, and assume that $U_v$ is maximal, i.e., $U_v\simeq \calo_{B_v}^\times$ and that $U^v$ is sufficiently small. A wonderful result of Cerednik and Drinfeld gives the following uniformisation
\begin{equation}\label{cduni}
	\widehat{\cals}_U\simeq G'(\Q)\backslash (\widehat{\Omega}\widehat{\otimes}\breve{\calo}_{F_v})\times\Z\times G'(\A_f^v)/U^v.
\end{equation}
For a detailed discussion of the Cerednik-Drinfeld uniformisation, we invite the reader to go through \cite{bc}.

Let $\X$ be the unique (up to isogeny) formal $\calo_{B_v}$-module of height $4$ over $k_v$. By Drinfeld moduli interpretation \cite[p.107]{bc}, $\widehat{\Omega}\widehat{\otimes}\breve{\calo}_{F_v}$ can be viewed as the moduli space of the (isomorphism classed of the) following objects:
\begin{itemize}
	\item a formal special $\calo_{B_v}$-module $X$ of height $4$;
	\item a height zero quasi-isogeny $\varrho\colon\X\rightarrow X_{k_v}$.
\end{itemize}
Note that, since this moduli interpretation is proved as usual via the representability of a moduli problem, this quasi-isogeny excludes a stacky situation. 
For a detailed description of the special fiber, see \cite{fms}.

Let $s$ be a geometric point of the special fiber of $\cals_U$ at $v$. Then by (\ref{cduni}) and Drinfeld moduli interpretation, $s$ corresponds to a formal $\calo_{B_v}$-module $X_v$. Then $s$ is $\mathit{supersingular}$ if $X_v$ is isogenous to the direct sum of two formal $\calo_{F_v}$-modules $Y_v$ of height $2$ and dimension $1$ such that $\End(Y_v)\simeq \calo_{B_v}$. It is not difficult to see, as in \cite[Lemma 2.9]{fms}, that every geometric point in this special fiber is supersingular. 
\\Nonetheless, there is a distinguished set of points encoding arithmetic and geometric properties in a manner similar to the supersingular locus in the unramified case. A geometric point $s$ is called $\mathit{superspecial}$ if the associated formal $\calo_{B_v}$-module $X_v$ is isomorphic to $Y_v\oplus Y_v$; this isomorphism is unique up to $\GL_2(\calo_{B_v})$-conjugation. For such a point $s$, its endomorphism ring is $\End(X_v)\simeq M_2(\calo_{B_v})$. The $\calo_{B_v}$-action on $X$ is given by
\[
\iota\colon\calo_{B_v}\hookrightarrow \End(X_v)\simeq M_2(\calo_{B_v}).
\]
The finite subscheme $\widehat{\cals}^{\text{ssp}}_{U,k}$ over $\Spec(k)$ consisting of superspecial points is thus called the $\mathit{superspecial}$ locus.

Geometrically, the irreducible components of the special fiber of $\widehat{\Omega}$ are projective lines $\p^1_k$ intersecting each other in the superspecial locus, which  is in bijection with the ordinary double points of $\widehat{\cals}_{U,k}$. Thus the complement of these singularities is the $\mathit{smooth}$ locus of $\widehat{\cals}_{U,k}$, which we denote by $\cals_{U,k}^{\text{sm}}$. For more details on this smooth locus, see \cite[Section 2]{fms}.

As for the supersingular points in Lemma \ref{adelicss}, also superspecial points admit an adelic interpretation.

\begin{lem}\label{adelicssp}
	We have that the set of  superspecial points on $\widehat{\cals}_{U,k}$ corresponding to the class of a fixed $\iota$ are in bijection with 
		\[
		G'(\Q)_0\backslash G'(\A_f^v)/ U^v
		\]
		
		where $G'(\Q)_0$ are the elements in the centralizer of $\iota(\calo_{B_v})$.
		
\end{lem}
\begin{proof}
	See \cite[Lemma 5.4.5]{sz}.
\end{proof}

\subsubsection{Eichler Orders associated to Supersingular and Superspecial Points.}\label{Grosslattice}

Next construction is a variant of \cite[pp.171-172]{gr} including the ramified setting.
\\

Let $v$ be a non-archimedean place of $F$ and consider the special fiber $\cals_v$. Let also $x\in\scal^{\text{CM}}$.
As usual, we consider $s$, namely the reduction of $x$ modulo $v$,  in the following cases:  
\begin{enumerate}
	\item  $v$ unramified in $B$ and inert in $K$;
	
	\item $v$ ramified in $B$ and ramified in $K$.
\end{enumerate}

In the unramified case, i.e., for $s$  in the supersingular locus, then we set $R_s=\End(s)$ as in  \cite[Section 4.1]{jk}.

On the other hand, if $v$ is ramified in $B$, i.e., for $s$ in the  superspecial locus,
we construct the desired Eichler order $R_s$ as follows.

Let $X_k$ correspond to the superspecial point $s$ in $\widehat{\Omega}_k$, whose quaternionic action is given by $\iota\colon \calo_{B_v}\hookrightarrow M_2(\calo_{B_v})$. Consider the maximal orders
\[
R_{\iota,v}:=\End_{\iota(\calo_{B_v})}(\X)\subset M_2(\calo_{F_v})=\End_{\calo_{B_v}}(X)
\]
where $M_2(\calo_{F_v})$ has rank $4$,
and
\[
\widehat{R}_f^v\subset \widehat{B}^v_f:=B\otimes\A_{F,f}^v.
\]
Let $\mathbb{B}'$ be the coherent\footnote{This terminology was introduced in \cite[p.3]{yzz}.} quaternion algebra over $\A_{F,f}$  and consider the Eichler order
\[
\widehat{R}:=R_{\iota,v}\times\widehat{R}_{f}^v\subset\mathbb{B}'.
\]
There exists a unique quaternion algebra over $B'$ such that $\mathbb{B}'\simeq B'\otimes \A_{F,f}$ and $\widehat{R}\subset B'\otimes \A_{F,f}$ so that
\[
R_s:=R\subset B'
\]
and $B'$ is ramified at $\tau$ and $\mathit{not}$ ramified at $v$. Note that this change of invariants takes place at $R_{\iota,v}$, the centralizer of the action of $\iota(\calo_{B_v})$.

\subsubsection{The Reduction Map}

Let $\bar{v}$ be a place of $K^{\text{ab}}$ above $v$, and denote by $\calo_{\bar{v}}$ and $k_{\bar{v}}$ its ring of integers at $\bar{v}$ and its residue field respectively.

The Shimura curve $\cals_U$ is proper over $\Spec(\calo_{F_v})$ both in the unramified and ramified cases,  so by the valuative criterion for properness we have 
\[
\scal_U(K^{\text{ab}})=\cals_U(K^{\text{ab}})\simeq \cals_U(\calo_{\bar{v}})\rightarrow \cals_U(k_{\bar{v}}).
\]
Since CM points are defined over $K^{\text{ab}}$, we obtain the following reduction map
\[
\text{red}_v\colon\cals_U^{\text{CM}}\longrightarrow\cals_U(k_{\bar{v}}).
\]
 
 \begin{lem}\label{redCM}
 	
	 The reduction of  CM points in the special fiber modulo $v$ lies  in
	 \begin{enumerate}
	 	\item the supersingular locus, if $B_v$ is unramified and $v$ is inert or ramified in $K$;
	 	\item the superspecial locus, if $B_v$ is ramified and $v$ ramifies in $K$;
	 	\item in the smooth locus, if $B_v$ is ramified and $v$ is inert in $K$.
	 \end{enumerate} In symbols,
	\[
	\text{red}_v(\CM)\subseteq
	\begin{cases} 
		&\cals^{\text{ss}}_{k},\;\;\; \text{for}\; B_v\simeq M_2(F_v), v \; \text{inert or ramified in}\; K;\\ 
		&\cals^{\text{ssp}}_{k},\;\;\text{for}\; B_v\;\text{ramified},\;\;v\;\text{ramified in}\;\; K;
		\\
		&\cals^{\text{sm}}_{k},\;\;\text{for}\; B_v\;\text{ramified},\;\;v\;\text{inert in}\;\; K.
	\end{cases}
	\]
\end{lem}
\begin{proof}
	See \cite[Lemma 3.1]{cv} for $B_v\simeq M_v(F_v)$ and \cite[Prop.2.17]{fms} for the ramified case.
\end{proof}

We now conclude by writing down the three sequences of measures whose limit will give us our equidistribution result.

Let the geometric point $s$ in the special fiber at $v$ lie in

\begin{itemize}
	\item the supersingular locus, for $B_v\simeq M_2(F_v)$;
	\item in the superspecial locus, for $B_v$ ramified.
\end{itemize}

For $\diamond\in\{\text{ss},\text{ssp}\}$, we define the following probability measures

\begin{equation}\label{measure}
	\mu^{\diamond}_{D,c}(s)=\dfrac{1}{\# \Gamma_{D,c}}\sum _{\substack{x\in\Gamma_{D,c} \\ \text{red}_v(x)=s}}\mathbf{1}_s(x),
\end{equation}

where $\mu^{\text{ss}}_{D,c}$ is defined over $\cals^{\text{ss}}_{U,k}$ and $\mu^{\text{ssp}}_{D,c}$ over $\widehat{\cals}_{U,k}^{\text{ssp}}$.

Lastly, let $s$ be a geometric point in the special fiber at $v$, for $B_v$ ramified, whose corresponding formal $\calo_{B_v}$-module is $\mathit{not}$ superspecial, so that it is supersingular, and it lies on one of the components  $\{c_1,\dots,c_n\}$ of the smooth locus $\cals_{U,k}^{\text{sm}}$. Exactly as above in (\ref{measure}), we then obtain the analogous measure $\mu_{\text{in},D,c}^{\text{ssp}}$.

\subsubsection{Liftings of Reductions Maps.}\label{liftred}

Consider the two following maps
\[
\uppi^\diamond\colon \Gr\rightarrow \cals^\diamond_k
\]
where once again $\diamond\in\{\text{ss},\text{ssp}\}$.
By Lemmata \ref{adelicgross}, \ref{adelicss} and \ref{adelicssp}, the $\pi^\diamond$'s are the natural projections defined by
\[
\begin{tikzcd}
	&T(\Q)\backslash G'(\A_f)/\widehat{R'}^\times
	\arrow[swap]{dl}{\uppi^{\text{ss}}}
	\arrow{dr}{\uppi^{\text{ssp}}}
	\\
	G'(\Q)\backslash G'(\A_f)/\widehat{R'}^\times  & & G(\Q)_0\backslash G'(\A_f)/\widehat{R'}^{v,\times}.
\end{tikzcd}
\]
The following construction allows to move from the reduction maps to the projections $\uppi^\diamond$. 
\\For $B_v$ unramified, by \cite[Section 3.2]{cj} there is a $G_K$-equivariant lifting $\theta_v$ such that the following diagram
\[
\begin{tikzcd}
	\scal^{\text{CM}} \arrow[dashed]{rr}{\theta_v} \arrow[swap]{dr}{\text{red}_v} & & \Gr \arrow{dl}{\uppi^\diamond} \\[10pt]
	& \cals^\diamond_k 
\end{tikzcd}
\]
is commutative. On the other hand, for $B_v$ ramified, then the construction of the lifting $\theta_v$ follows almost verbatim the unramified case.
\\

Let now $R$ be an Eichler $\calo_F$-order in $B$ given by the maximal orders $R'$ and $R''$, and consider the Shimura curve of level $\widehat{F}^\times\widehat{R}^\times$. By the interpretation of CM points as an adelic double quotient in Lemma \ref{adelicCM} we obtain the two maps
\[
\begin{tikzcd}
	&T(\Q)\backslash G(\A_f)/\widehat{R}^\times
	\arrow[swap]{dl}{}
	\arrow{dr}{}
	\\
	T(\Q)\backslash G(\A_f)/\widehat{R'}^\times  & & T(\Q)\backslash G(\A_f)/\widehat{R''}^{\times}.
\end{tikzcd}
\]
Given a $z\in \scal^{\text{CM}}$, we denote its image under these maps as $z'$ and $z''$ respectively. For $z=[g]$, then $K\cap g\widehat{R'}^\times g^{-1}$ is an $\calo_F$-order in $K$ whose conductor, denoted $c(z')$, is an integral $\calo_F$-ideal coprime to $\Ram_f(B)$. Symmetrically, $c(z'')$ is the conductor coming from $R''$. We thus obtain the ($\mathit{coarse}$) $\mathit{conductor}$ map
\[
\mathbf{c}\colon z\in\scal^{\text{CM}}\mapsto \mathbf{c}(z):=c(z')\cap c(z'').
\]
In a similar way, by Lemma \ref{adelicgross} we can define the analogous conductor map $\mathbf{c}'$ for $G'$. 
\begin{lem}\label{lem2.7}
	Suppose that the conductor $c$ is coprime to both $v$ and $\gotn$, and that $v$ is coprime to $\gotn$. Let $s$ be a geometric point in the special fiber at $v$, which lies in the supersingular locus if $B_v$ is unramified and in the superspecial locus if $B_v$ is ramified. Accordingly, let $R_s$ as in Section \ref{Grosslattice}. Then we have
	\[
	\mu^\diamond_{D,c}(s)=\dfrac{h(\calo_{D,c},R_s)}{2^{\omega(\gotn)+\varepsilon_\diamond}\#\picard(\calo_{D,c})},
	\]
	for $\diamond\in\{\text{ss},\text{ssp}\}$ and $\varepsilon_{\text{ss}}=1$ and $0$ otherwise.
\end{lem}
\begin{proof}
We recall that $h(\calo_{D,c},R_s)$ denotes the number of equivalence classes modulo $R_s^\times$ of optimal embeddings $\calo_{D,c}\hookrightarrow R_s$. By Lemma \ref{numberGalorbits} we have that there are $2^{\omega(\gotn)}$ Galois orbits of CM points of conductor $c$ and discriminant $D$, and that each such orbit has cardinality $\#\picard(\calo_{D,c})$.

By \cite[Cor.3.2]{cj}, we have that CM points of conductor $c$ and discriminant $D$ reducing to $s$ are in bijection with the equivalence classes of the aforementioned optimal embeddings. In particular we have
\begin{equation}\label{2.7}
	2^{\varepsilon_\diamond}\cdot\#(\mathbf{c}^{-1}(c)\cap \text{red}_v^{-1}(s))=\#(\mathbf{c}'^{-1}(c)\cap\uppi^{-1}(s))
\end{equation}
where in the superspecial case $\varepsilon_{\text{ssp}}=0$ by the  ramification of $v$ in $K$ following \cite[Section 3.4]{cj}. By Lemma \ref{numberGalorbits}, we obtain
\[
h(\calo_{D,c},R_s)= 2^{\omega(\gotn)+\varepsilon_\diamond}\cdot\#(\text{red}_v^{-1}(s)\cap\Gamma_{D,c}).
\]
Since the summation in (\ref{measure}) gives $\#\{z\in\Gamma_{D,c} : \text{red}_v(z)=s\}$, this shows that
\[
\mu_{D,c}^\diamond(s)=\dfrac{\#\{z\in\Gamma_{D,c} : \text{red}_v(z)=s\}}{\#\Gamma_{D,c}}=\dfrac{h(\calo_{D,c}, R_s)}{2^{\omega(\gotn)+\varepsilon_\diamond}\#\picard(\calo_{D,c})}.
\]
\end{proof}

\begin{rmk}
	In \cite[Thm.3.1]{cj}, the lifting $\theta_v$ does not induce a bijection, but rather a $\kappa$-to-$1$ surjection, where $\kappa$ has the cardinality of the Galois group $\Gal(K[c]/K[c_v])$ of the ring class field $K[c]$ as a factor, where $c_v$ is the prime-to-$v$ part of $c$. Therefore, our hypothesis that $c$ is coprime to $v$ implies the triviality of such a Galois group.
	The remaining factor of $\kappa$, which grosso modo is the cardinality of the set of $K_v^\times$-orbits of pairs of vertices of the Bruhat--Tits tree of $\PGL_2(F_v)$, is again $1$  by our coprimality hypothesis combined with \cite[Lemma 2.1(i)]{cj}.
\end{rmk}

\section{Equidistribution and Quadratic Forms}

\subsection{Automorphic Forms and Representations of Half-Integral Weight}

\subsubsection{Hilbert Modular Forms of Half-Integral Weight}

In this section we begin with some notions on Hilbert modular forms of weight $k=(k_1,...,k_d)\in\frac{1}{2}\Z^d_{\ge 0}$. We set $\Hom(F,\R)=\{\tau_1,\dots,\tau_d\}$ and denote $z=(z_1,\dots,z_d)\in\mathscr{H}^d$. We invite the reader to go through \cite{shi} as a basic reference.

\begin{rmk}
In what follows the reductive group $\GL_2$ has nothing to do with the quaternionic setting of the Shimura curve (\ref{anShim}), since the equidistribution result we aim to makes an auxiliary use of these automorphic forms.
\end{rmk}

Let $k$ be as above and let $m\in\Z_{\ge 0}^d$. Let also $U$ be a compact open subgroup of $\GL_2(\A_{F,f})$. Denote by $e(x)$ the standard additive character of $\A_F$, i.e., $e(x)=\prod_v e_v(x_v)$ where $e_v(x_v)=\exp(2\pi i\{tr_{F_v/\Q_p}(x_v)\}_p)$, for $\{a\}_p$ is the $p$-fractional part of $a\in\Q_p$. 
\\A complex {\em (automorphic) Hilbert modular form} of weight $k$ and level $U$  is a (non necessarily holomorphic) function
\[
f\colon \GL_2(\A_F)\rightarrow \C
\]

such that:

\begin{enumerate}
	\item for all $g\in\GL_2(\A_F)$, $\gamma\in\GL_2(F)$, $h\in U$,%C_\infty^+$,
	\[
	f(\gamma g h)= J_k(h_{\infty},z)f(g)
	\]

	where $J_k$ is the automorphy factor defined combining \cite[3.1b p.777]{shi} and \cite[1.4 p.770]{shi} and $h_\infty$ is the archimedean part of $h$;
	
	\item  there is a Whittaker--Fourier expansion
	%\begin{equation}\label{fwexp}
	%	f \left( \begin{bmatrix}
	%		y & x \\
	%		 &  1
	%	\end{bmatrix}\right)= |y|\sum_{\delta\in F^\times}W_{f,\delta}(y)(Y)q^{\delta} 
	%\end{equation}
%\end{enumerate}
\begin{equation}\label{fwexp}
f(g)=C_f(g)+\sum_{\delta\in F^\times}W_f\left( \begin{bmatrix}
	\delta & \\ 
	&1
\end{bmatrix}g\right) 
\end{equation}

where $C_f(g)$ is a constant term and $W_f$ is a Whittaker function defined as
\[
W_f(g)=\int_{\A/F}f\left( \begin{bmatrix}
	1 &x\\
	&1
\end{bmatrix}g\right)e(-x) dx
\] 

where $dx$ denotes the obvious Haar measure induced by the adelic one.
\end{enumerate}

We say that $f$ is $\mathit{cuspidal}$ if $C_f(g)$ is identically zero.

If $f$ is holomorphic they vanish unless $y_\infty >0$. If this is the case then the Whittaker functions have the following expression
\[
W_f\left( \begin{bmatrix}
	y & x\\
	  & 1
\end{bmatrix}\right) =\tilde{a}(f,y)e_\infty(i y_\infty)e(x)
\]
involving a function $\tilde{a}(f, y)$ of $y\in\A_f^\times$ %(and the standard additive character of $\A_F$) 
which are called $\mathit{Whittaker}$--$\mathit{Fourier}$ $\mathit{coefficients}$ of $f$. For a finite idele $y=(y_v)_v\in\A_{F,f}^\times$, then one classically obtain a fractional ideal
\begin{equation}\label{idele-ideal}
\mathfrak{y}=y\widehat{\calo}_F\cap F=\prod_v \gotp ^{\val_v(y_v)},
\end{equation}
where $\gotp$ corresponds to the finite place $v$ of $F$.
The Whittaker--Fourier coefficients of $f$ can be rewritten in terms of a function $a(f,\mathfrak{y})$ on the fractional ideals of $F$ vanishing on the non-integral ideals. %We can thus think of $a(f,\mathfrak{y})$ as a two parameters family of coefficients as  $%The values of such a function is called $\mathit{Fourier}$ $\mathit{coefficienst}$ of $f$. 
For a more detailed discussion, see \cite[Prop.3.1.2]{sz2} and \cite[Section 4.1.2]{rag}.

Let $\omega\colon \A_F/F\rightarrow \C$ be a finite order character. We recall that $f$ is of character $\omega$ if $f(cg)=\omega(c)f(g)$ for all $c\in \A_F^\times$. We denote by $M_k(U,\omega)$ the space of holomorphic Hilbert modular forms of level $U$ and character $\omega$, and by $S_k(U,\omega)$ its subspace of cuspidal forms.

Let now the map $\gota\colon \GL_1(\A_{F,f})\rightarrow \GL_2(\A_{F,f})$ be defined by $a\mapsto \begin{bmatrix}
	a & 
	\\
	 & 1
\end{bmatrix}$. We consider, following \cite[(2.1.3)]{dis2}, the $q$-expansion map  on $M_k(U,\omega)$
\begin{equation}\label{expansionprinciple}
f\mapsto (\delta\to \tilde{a}(f,\delta y)).
\end{equation}
This map is injective, by the $q$-expansion principle (see \cite[Prop.2.1.1]{dis2}), for $\delta y\in \gota^{-1}(U)$.
\\

Following \cite{shi}, we also recall the classical definition of Hilbert modular forms of half-integral weight. For the comparison with the automorphic approach, see, for instance, \cite[Chapter 3]{gar}.
\\Let $f$ be a holomorphic function on $\mathscr{H}^d$, and let $\gamma=(\gamma_1,\dots,\gamma_d)\in\GL^+_2(\R)^d$ such that $\gamma_i$ has $(2,1)$ and $(2,2)$ entries $c_i$ and $d_i$ respectively. For $z\in\mathscr{H}^d$ and $k=(k_1,\dots,k_d)\in\frac{1}{2}\Z_{\ge0}$, we consider the slash operator
\[
f|_k\gamma(z)=f(\gamma z)\prod_{i=1}^d\det(\gamma)^{-1/2}(c_iz_i+d_i)^{-k_i}.
\]
Let $U$ be a congruence subgroup of $\GL_2^+(F)$, and let $\omega$ be a character of $U$ of finite order. We then say that $f$ is a {\em (classical) Hilbert modular form} of weight $k$, level $U$, character $\omega$ if
	\begin{itemize}\label{classicalhilbert}
	\item $f|_k\gamma=\omega(\gamma)f$ for all $\gamma\in U$;
	\item the Fourier expansion $f|_k\gamma=\sum_{\delta\in F}c_\delta(\gamma)\exp(2\pi i \Tr(\delta z))$ has $c_\delta(\gamma)=0$ unless $\delta=0$ or $\delta\in F^+$, for all $\gamma\in\GL_2^+(\R)$.
	\end{itemize}

\subsubsection{Metaplectic Covers}
In this section we introduce a $2$-fold cover of $\GL_2$, as in \cite{gps}. Consider the local metaplectic group $\widetilde{\SL}_2$ given by
\[
\widetilde{\SL}_2(F_v)=\begin{cases}
	& \SL_2(F_v)\times\Z/2\Z, \;\;\text{for v archimedean}\\
	& \text{the non-split central extension of}\; \SL_2(F_v)\; \text{by}\; \Z/2\Z, \; \text{otherwise}.
\end{cases}
\]
which is a double cover of $\SL_2$. Note that the central extension of $\SL_2(F_v)$ is determined by the non-trivial element of the continuous group cohomology  $\text{H}^2_{\text{cont}}(\SL_2(F_v),\Z/2\Z)\simeq\Z/2\Z$.

The group
\[
\left\lbrace \begin{bmatrix}
	a & \\
	& 1
\end{bmatrix} : a\in F_v^\times\right\rbrace 
\]
gives an action of $F_v^\times$ on $\SL_2(F_v)$ by conjugation, which uniquely lifts to an automorphism of $\widetilde{\SL}_2(F_v)$. We denote by $\widetilde{\GL}_2(F_v)$ the semi-direct product of $\widetilde{\SL}_2(F_v)$ and $F_v^\times$. Hence this locally compact group fits in the following short exact sequence
\[
1\longrightarrow\Z/2\Z\longrightarrow \widetilde{\GL}_2(F_v)\longrightarrow \GL_2(F_v)\longrightarrow 1.
\]
The center of $\widetilde{\GL}_2(F_v^\times)$ is $Z^2\times \Z/2\Z$, where 
$Z=\left\lbrace \begin{bmatrix}
	a & \\
	& a
\end{bmatrix}: a\in F_v^\times\right\rbrace $.
\\

Globally, the metaplectic group  is defined  as the quotient  $$\widetilde{\GL}_2(\A_F)=\prod'_v\widetilde{\GL}_2(\calo_{F_v})/\widetilde{Z},$$ where $\widetilde{Z}=\{\prod_v\epsilon_v\in\prod_v\Z/2\Z \;|\; \epsilon_v=0 \;\text{for an even number of}\; v\}$. 
\\

We conclude this section  by introducing a dictionary between certain Hilbert modular forms and automorphic representations of $\widetilde{\GL}_2(\A)$, following  \cite{rag} and \cite{gel}.

We say a cusp form $f$ is $\mathit{primitive}$ if it is a newform, an eigenform for all Hecke operators as defined in \cite[Section 4.1.3]{rag}, and normalized with conductor equal to $1$.

\begin{lem}\label{dictionary}
Let $\omega$ be a character of $(\calo_F/\gotn)^\times$ and let $\tilde{\omega}$ denote its adelization, i.e., a character of $\A^\times_F/ F^\times$ induced from $\omega$. Let also	$D_i$ be the discrete series representation of $\widetilde{\GL}_2(\R)$ of lowest weight $k_i\in\frac{1}{2}\Z$ with central character $a\mapsto |a|^{k_i}$.
We have a bijection 
	\[
	\Set{\begin{array}{c}
			\text{primitive  cusp}
			\text{ Hilbert modular forms} \\
			\text{of  weight $k=(k_1,\dots,k_d)\in\frac{1}{2}\Z^d$,}\\ \text{level $\gotn$, 
			nebentypus $\omega$}\end{array}}
	\longleftrightarrow
	\Set{
		\begin{array}{c}
			\text{cuspidal automorphic representations of $\widetilde{\GL}_2(\A_F)$ } \\
			\text{ of conductor $\gotn$, 
					central character $\tilde{\omega}$ }\\
			\text{whose representation at infinity is $\otimes_{j=1}^n D_{k_j-1}$}
	\end{array}}.
	\]

\end{lem}
\begin{proof}
	In view of the metaplectic theory developed in \cite{gel}, the proof consists of an adaptation of	\cite[Thm.1.4]{rag}, so we just summarize the above correspondence. In particular, at the archimedean places, the result follows from \cite[Section 4.1]{gel} which consists of the metaplectic counterpart of Langlands' classification for $\GL_2(\R)$.
	
	Given a primitive Hilbert cusp form $f$ of half-integral weight $k$, we consider the space $H_f$ spanned by right $\widetilde{\GL}_2(\A_F)$-translations of $f$. We thus obtain a representation $\Pi(f)$ on $H_f$ which occurs in the regular representation on the cusp forms. The representation $\Pi(f)$ is irreducible; this can be proved verbatim as in \cite[Thm.4.7]{rag}.
	
	On the other hand, let $\Pi$ be a cuspidal automorphic representation as in the left-hand side of the bijection and let $V_\Pi$ be its representation space. The Whittaker model of $\Pi$ is isomorphic to $V_\Pi$ and for  specific choices of vectors in the local Whittaker model of $\Pi_v$ one can determine uniquely an element $f$ in $V_\Pi$ giving the desired Hilbert cusp form.
\end{proof}

\subsubsection{Weil Representations}\label{weilrepr}

%Theta functions arise in a represention theoretic manner via the Weil repr.
%Let $\psi_v$ denote the local standard additive character, and $\psi=\prod_v\psi_v$ the adelic additive character on $\A_F/F$. 
For each $v$, let here $\delta_v\in F_v$ be the conductor of $e_v$. For $v$ finite, $\delta_v\calo_{F_v}$ is the maximal fractional ideal over which $e_v$ is trivial. On the other hand, for $v$ a real place, $e_v(x)=\exp(2\pi i \delta_vx)$. Moreover, we have that the absolute value of $\delta:=\prod_v\delta_v\in \A_F^\times$ equals the discriminant of $F$; in symbols, $|\delta|=d_F$.
\\

Let $\mathtt{S}(F_v^n)$ denote the space of $\C$-valued Schwartz--Bruhat functions. For $f\in \mathtt{S}(F^n_v)$, we define its Fourier transform as in \cite[p.36]{gel} by
\[
\widehat{f}(x)=\int_{F^n_v} f(y)e_v(2xy)dy
\]
where $dy$ is the normalized Haar measure so that $\widehat{\widehat{f\,}}(x)=f(-x)$.

Let $Q$ be a quadratic form on $F^n$, and let $\gamma_v$ be a eight root of unity, and $\gamma_{v,a}$ its translate by $a\in F_v^{\times}$. The $\mathit{local}$ Weil representation $r(e_v)$ is the unique representation of $\widetilde{\SL}_2(F_v)$ that can be realized on $\mathtt{S}(F^n_v)$ by the following formulae:
\begin{itemize}
	\item $r(e_v)\begin{bmatrix}
		& 1\\
		-1&
	\end{bmatrix}f(x)=\gamma_v\widehat{f}(x)$;
	\item  $r(e_v)\begin{bmatrix}
		1 & t\\
		& 1
	\end{bmatrix}f(x)=e_v(tQ(x))f(x)$ for $t\in F_v$;
	\item $r(e_v)\begin{bmatrix}
		a & \\
		& a^{-1}
	\end{bmatrix}f(x)=|a|^{1/2}_v\dfrac{\gamma_v}{\gamma_{v,a}}f(x)$;
	\item $r(e_v)(\varepsilon)f(x)=\varepsilon f(x)$ for $\varepsilon$ is a second root of unity.
\end{itemize}

Note that these formulae uniquely define the Weil representation via the Bruhat decomposition of $\SL_2$.
\\

Globally, for a totally real field $F$ and a quadratic space $(V,Q)$ over $F$, consider the non-trivial character $e$ on $\A_F/F$. Then the Weil representation is obtained as the restricted tensor product of the local Weil representations $r(e_v)$'s.
\\

Consider now the quadratic space $(V,Q)$ over $F_v$, for $Q$ a $n$-ary form. By \cite[Ex.2.21]{gel} each quadratic form $Q$ corresponds to a Weil representation. Whenever in need to emphasize this we will denote the Weil representation corresponding to $Q$ by $r_Q$.

\subsubsection{Theta-representations} We now focus on the unary case. The  local Weil representation $r(e_v)$ decomposes as a direct sum of two irreducible subrepresentations on the subspaces of  even and  odd Schwartz--Bruhat functions. Thus, for $\chi_v$ character of $F_v^\times$, we can tensor with the even or odd part of $r(e_v)$, according to the parity of $\chi_v$. By inducing this representation to $\widetilde{\SL}_2(F^n_v)$ we obtain an irreducible admissible representation $r(\chi_v)$ independent of $e_v$, which is unramified whenever $\chi_v$ is unramified, i.e., for all but finitely many $v$.

We now extend $r(e)$ to a representation of $\widetilde{\GL}_2(F^n_v)$, so to remove the dependence on $e_v$.

Consider the pullback of $\{g\in\GL_2(F^n_v) : \det(g)\in(F_v^\times)^2 \}$ in $\widetilde{\GL}_2(F^n_v)$ taking the form 
\[
\widetilde{G}^*:=\widetilde{\SL}_2(F_v)\rtimes \left\lbrace  \begin{bmatrix}
	1 & \\
	& a^2
\end{bmatrix}: a\in F_v^\times \right\rbrace .
\]
We then extend $r(e_v,\chi_v)$ to $\widetilde{G}^*$ by setting
\[
r(e_v,\chi_v)\begin{bmatrix}
	1 & \\
	& a^2
\end{bmatrix}f(x)=\chi_v(a)|a|^{-1/2}_vf(a^{-1}x).
\]

Moreover, we have that $r(e_v,\chi_v)$ is an irreducible, admissible representation of $\widetilde{G}^*$.

%follow Jacquet-Langlands pp 125-130 and 146-147.

Inducing up $r(e_v,\chi_v)$ to $\widetilde{\GL}_2(F^n_v)$ produces a representation $r(\chi_v)$, which is irreducible and admissible and independent of $e_v$ (see \cite[1.3, p.150]{gps1}).
\\

We conclude by the global setting. Recall that an automorphic representation of half-integral weight is an irreducible admissible representation of $\widetilde{\GL}_2(\A_F)$ which is isomorphic to a subrepresentation of the space of automorphic forms on $\widetilde{\GL}_2(\A_F)$. 

For every character $\chi$ of $\A^\times/F^\times$, 
the (global) theta-representation is 
\[
r(\chi)=\otimes_v'r(\chi_v),
\]
namely the restricted tensor product of the local Weil representations. This gives an automorphic representation of half-integral weight. Moreover, if there exists at least a place $v$ of $F$ such that $\chi$  is odd (i.e., $\chi_v(-1)=-1$), then $r(\chi)$ is cuspidal (see \cite[Prop.8.1.1]{gps1}).

Let us introduced two important technical definitions. We say that an irreducible admissible representation $\overline{\pi}_v$ of $\widetilde{\GL}_2(F_v)$ is $v$-$\mathit{distinguished}$ if its Whittaker model, which always exists, is unique. Globally, an irreducible admissible representation $\overline{\pi}=\otimes'\overline{\pi}_v$ of $\widetilde{\GL}_2(\A_F)$ is distinguished if it is $v$-distinguished at all places $v$.
An automorphic representation of $\widetilde{\GL}_2(\A_F)$ is called $\mathit{genuine}$ if it does not factor through $\GL_2(\A_F)$.

\begin{lem}\label{thetarepr}
	There is a bijection 
	\[
	\Set{\begin{array}{c}
			\text{Hecke characters of $F$}
	\end{array}}
	\longleftrightarrow
	\Set{
		\begin{array}{c}
			\text{genuine, distinguished, automorphic}\\
			\text{ representations of $\widetilde{\GL}_2(\A_F)$}
	\end{array}}
	\]	
	attaching to every Hecke character $\chi$ the theta-representation $r(\chi)$.
\end{lem}
\begin{proof}
	See \cite[Thm.A p.147]{gps1}.
\end{proof}

\begin{rmk}\label{Unary}
	We will make extensive use of the fact that $r(\chi)$ generalizes the construction that associates to a Dirichlet character $\chi$ on $\Z/N\Z$ the classical theta series 
	\[
	\theta(z,\chi)=\sum_{n\ge 1}\chi(n)nq^{n^2z}.
	\]
	It is well known that this gives a primitive modular form of half-integral weight, level $4N^2$ and character $\chi$.
	Moreover, for $t$ squarefree we have that the subspace $\calu(N,\chi)$ of cuspidal modular form of weight $3/2$ whose Shimura lift is $\mathit{not}$ cuspidal is spanned by $\theta(z,\chi)$'s (see, for instance, \cite[p.365]{hanke}).
\end{rmk}

\subsubsection{Eisenstein Series}
%both are automorphic forms

Let $B$ denote the parabolic group of $\GL_2$, i.e., its Borel subgroup, with Levi decomposition $B=NM$ for $N$ being the unipotent subgroup and $M$ the Levi subgroup. Denote also by $\GL_2(\A_F)^1$ the elements of $\GL_2(\A_F)$ whose determinant has norm $1$, and by $M(\A_F)^1$ the matrices in the Levi subgroup whose entries are of norm $1$.
\\

For $\varphi\in \calc^\infty_c(N(\A_F)M(\A_F)\backslash\GL_2(\A_F)^1)$ we define the $\mathit{pseudo}$-$\mathit{Eisenstein}$ series as
\[
\Psi(\varphi,g)=\sum_{\gamma\in B(F)\backslash \GL_2(F)}\varphi(\gamma g).
\]
Note that $\Psi(\varphi,g)$ is locally finite\footnote{I.e.,  for $g$ in a fixed compact of $\GL_2(\A_F)$, the sum defining $\Psi(\phi, g)$ has only finitely many summands.}, so it converges and defines and element of $\calc^\infty_c[\GL_2(\A_F)^1]$ (see \cite[2.7.1]{gar}).

Let now $\chi$ be a character of $M(F)\backslash M(\A_F)^1$. Note that any character of $[M(\A_F)^1]$ is of the form $\nu^s\chi$ where $\nu^s\colon \begin{bmatrix}
	\delta(y) &\\
	& 1
\end{bmatrix}\mapsto |y|^s$, where $\delta\colon (0,\infty)\rightarrow \A_F^\times$ is the diagonal embedding into the archimedean component of the ideles by $\delta(y)=(\dots,y^\frac{1}{d_v},\dots)$ with $d_v$ the local degree at $v$. 
We thus consider the following unramified principal series representation of $\GL_2$ given by parabolic induction 
\[
I_{s,\chi}=\{f\in\calc_c^\infty(N(\A_F)M(F)\backslash \GL_2(\A_F)^1) : f(nmg)=(\nu^s\chi)(m)f(g)\;\text{for all}\; n\in N(\A_F), m\in M(\A_F) \}.
\]

The $\mathit{Eisenstein}$ series is defined as
\[
E(f,g)=\sum_{\gamma\in B(F)\backslash \GL_2(F)}f(\gamma g),
\]
for $f\in I_{s,\chi}$.
Note that the Eisenstein series is not a $\call^2$-function\footnote{The obstruction to this consists of the constant term.}, while the pseudo-Eisenstein series is. Moreover, the pseudo-Eisenstein series can be expressed  in terms of Eisenstein series by integral converging uniformly absolutely on compacts in $\GL_2(F)\backslash\GL_2(\A_F)^1$. For the pointwise formula, see \cite[Thm.2.11.1]{gar}.

\subsubsection{Automorphic Shimura Lifting}\label{automorphicShimura}

In the local case, we have that an irreducible admissible representation $\pi_v$ of $\GL_2(F_v)$ is the local Shimura image of $\overline{\pi}_v$, in symbols, $\shi(\overline{\pi}_v)=\pi_v$, if
\begin{itemize}
	\item for $a\in F_v^\times$, we have $$\omega_v(a)=\overline{\omega}_v(a^2),$$ 
	where $\omega_v,\overline{\omega}_v$ are the central characters  of $\pi_v$ and $\overline{\pi}_v$ respectively;
	\item for any character $\chi$ of $F_v^\times$, equalities between automorphic $L$-functions and $\varepsilon$-factors hold respectively as in \cite[Section 7.1]{gps}.
\end{itemize}

The Shimura lift of $\overline{\pi}$, if it exists, is unique\footnote{Moreover, as suggested by its notation, it does not depend on $e_v$, while the $\varepsilon$-factor does.}.
For a $L$-functions-free approach to the Shimura lifting, see \cite[Section 5]{fli}.
\\

As usual, we define the global Shimura lifting by putting the local pieces together. Let $\overline{\pi}=\otimes'_v\overline{\pi}_v$ be an irreducible admissible representation of $\widetilde{\GL}_2(\A_F)$.
We say that the irreducible admissible representation $\pi$ of $\GL_2(\A_F)$ is the $\textit{Shimura}$ $\textit{image}$ of $\overline{\pi}$, in symbols, $\shi(\overline{\pi})=\pi$, if, for every place $v$, the same holds locally, i.e., $\shi(\overline{\pi}_v)=\pi_v$.
\\

Let now $\chi=\prod_v\chi_v$ be a Hecke character of $F$, and consider the case $\overline{\pi}=r(\chi)$. By \cite[Thm.15.1]{gps} we obtain that $\shi(r(\chi_v))$ is  $1$-dimensional and unramified for almost every place $v$, and so
\begin{equation}
	\shi(r(\chi))=\otimes'_v\shi(r(\chi_v))
\end{equation}
is automorphic but not cuspidal. In particular, $\shi(\overline{\pi})$ is cuspidal if and only if $\overline{\pi}$ is not in the image of the correspondence of Lemma \ref{thetarepr}.

In the archimedean case $F_\infty$, then $\overline{\pi}_\infty$ is a discrete series representation of lowest weight $k/2$ and so $\shi(\overline{\pi}_\infty)$ corresponds to a discrete series representation of lowest weight $k-1$ (see \cite[Prop.4.8]{gel}).
\\

Lastly, we extend the Serre--Stark theorem (and the considerations of Remark (\ref{Unary})) to this automorphic setting as in \cite{gps1}. This means that the theta representations $r(\chi)$'s exhaust certain automorphic representations of weight $1/2$, which are defined to be such that there exists one archimedean place such that $\overline{\pi}_\infty$ is the even part of the Weil representation.

\begin{lem}
	Suppose that $\overline{\pi}$ is a cuspidal representation of $\widetilde{\GL}_2(\A_F)$ of weight $1/2$. Then there exists a Hecke character $\chi$ of $F$ such that $\overline{\pi}=r(\chi)$.
\end{lem}
\begin{proof}
	This follows as a consequence of the surjective part of Lemma \ref{thetarepr} and the Shimura lift.
\end{proof}

\subsubsection{Langlands Spectral Decomposition}

In this section we mainly follow \cite{arth} and \cite{lab} as references.

Let $\call^2[\GL_2]$ denote the Hilbert space of functions on $[\GL_2]$ that are square-integrable with respect to the natural measure, and denote  by $R$ the right regular representation of $\GL_2(\A_F)$ on $\call^2[\GL_2]$. This representation is isomorphic to the completion of a discrete sum of unitary irreducible representations of $\GL_2(\A_F)$.

Consider $\GL_2(\A_F)^1=\{g\in\GL_2(\A_F) : |\det g|_\A=1\}$ where $|\cdot|_\A$ is the adelic norm. Denote by $\call$ the space $\call^2(\GL_2(F)\backslash \GL_2(\A_F)^1))$. We want to introduce the spectral decomposition of $\call$ under $R$.
\\

Let us introduce the $\mathit{discrete}$ spectrum of $\call$ which we call $\call_{\text{disc}}$. It is defined as a Hilbert sum of irreducible subspaces of $\call$, i.e., the closed subspace generated by the irreducible subrepresentations of $\GL_2(\A_F)^1$ in $\call$. Thus $\call_{\text{disc}}$ is the part of $\call^2[\GL_2(\A_F)^1]$ whose spectral decomposition looks like the decomposition of $R$. The complement of $\call_{\text{disc}}$ in $\call$ is called the $\mathit{continuous}$ spectrum, and denoted by $\call_{\text{cont}}$.

Let $B$ be the Borel subgroup of $\GL_2$ and $N$ its unipotent subgroup. Consider the square-integrable function on $[N]=N(\A_F)\cap \GL_2(F)\backslash N(\A_F)$ defined for almost all $x\in\GL_2(\A_F)^1$ as
\[
n\mapsto \varphi(nx)
\]
for $\varphi\in\call$.

We define the $\mathit{constant}$ $\mathit{term}$ along $B$ the function $\varphi_B$ defined as
\[
\varphi_B(x)=\int_{[N]}\varphi(nx)dx,
\]

where $dx$ is the Haar measure on $[N]\simeq F\backslash\A_F$. The function $\varphi$ is said to be $\mathit{cuspidal}$ if its constant term along $B$ vanishes almost everywhere. We thus  define 
$\call_{\text{cusp}}$
to be the subspace of cuspidal functions $\varphi\in\call$. It can be decomposed as the Hilbert sum of unitary irreducible representations of $\GL_2(\A_F)^1$ with finite multiplicities\footnote{In fact, its multiplicities are either $0$ or $1$.}.

Next Lemma yields a construction of the orthogonal complement of $\call_{\text{cusp}}$ in $\call$.

\begin{lem}\label{3.17}
	The $\call^2$-closure of the set of pseudo-Eisenstein series spans the orthogonal complement of $\call_{\text{cusp}}$.	
\end{lem}
\begin{proof}
	See \cite[2.7.2]{gar}.
\end{proof}

The orthogonal complement of $\call_{\text{cusp}}$ in $\call_{\text{disc}}$ is called the $\mathit{residual}$ spectrum, and denoted by $\call_{\text{res}}$. It is obtained by residues of Eisenstein series. In our situation we have that $\call_{\text{res}}$ decomposes as the direct sum of $\chi\circ\det$, for $\chi$ ranging among the characters of $\A_F^1/F$.

\begin{prop}\label{spectraldec}
	We have the orthogonal decomposition
	\begin{equation}
		\call=\call_{\text{cusp}}\oplus\call_{\text{res}}\oplus\call_{\text{cont}}.
	\end{equation}
\end{prop}
\begin{proof}
	This statement follows from \cite[Thms 2.10.3, 2.10.4]{lab} and \cite[Prop. 2.5.10]{lab}.
\end{proof}

\subsection{Ternary Theta Series}

\subsubsection{Ternary Quadratic Forms}

Let $V$ denote a $3$-dimensional vector space over $F$. An $\calo_F$-valued quadratic form $Q$ over $\calo_F$ is $\mathit{primitive}$ if the ideal generated by its $\calo_F$-values is $\calo_F$. Moreover $Q$ $\mathit{represents}$ $\delta$ if $\delta\in Q(\calo_F)$. We also denote the orthogonal group of a quadratic space $(V,Q)$ by $O(V)$, which consists of the invertible linear transformations $L\in\End(V)$ such that $Q(Lv)=Q(v)$ for all $v\in V$. Moreover, rotations form the special orthogonal group $\SO(V)$, i.e., isometries of determinant $1$. We also recall that a non-zero vector $v$ is called $\mathit{anisotropic}$ if $Q(v)\neq 0$. We also say that $Q$ is $\mathit{positive}$ $\mathit{definite}$ if $Q$ represents only positive values. For such a form $Q$, essential to our study are the $\mathit{representation}$ numbers
\[
r(Q,\delta):=\#\{x\in\calo_F^3 : Q(x)=\delta \}.
\]
Note that  assuming the positive definiteness is essential, since it ensures the finiteness of the representation numbers.

We can now define the theta series of $Q$ and its Fourier expansion. For $\delta\in F$ and $z=(z_1\dots,z_d)\in\mathscr{H}^d$, we denote by $q^{\delta}:=\exp(2\pi i\sum_{i=1}^d\tau_i(\delta)z_i)$. 

Let $Q$ be a non-degenerate\footnote{I.e., $\det Q\neq 0$.} positive definite ternary quadratic form on $F^3$. Then
\begin{equation}\label{half}	\uptheta_Q(z)=\sum_{v\in \calo_{F}^3}q^{Q(v)}=\sum_{\delta\in F}r(Q,\delta)q^\delta
\end{equation}
is a (classical) Hilbert modular form of  weight $3/2$ for $\SL_2(\calo_F)$. See \cite[p.154]{gar} or \cite[Section 4]{shi} for a proof.
%lemma 4.3
%\cite[Cor. 2.3.2]{hanke2} for a proof.
%\\

%A quadratic space $(V,Q)$ can be enriched with a multiplicative law given by its Clifford algebra $\clf(Q)$, which we recall being the quotient of the tensor algebra $T(V)$ by the ideal generated by the squaring relations $Q(v)=v^2$. Note that $\clf(Q)$ inherits a $\Z/2\Z$-grading from $T(V)$, and we denote by $\clf_0(Q)$ the algebra of elements of even degree.

From now on, let $B^0$ denote the set of trace zero element in the quaternion algebra $B$. To $B$, one can associate the $4$-dimensional quadratic space $(V,Q)=(B,\nr)$.

%\begin{lem}
%	Assume $\calo_F$ is a principal ideal domain. Then we have
%	\[
%	r(Q,\delta)=\#\{I\subset\calo : N\cdot \nr(I)=\delta\}.
%	\]
%\end{lem}
%\begin{proof}
%By \cite[Thm.22.1.1]{voi} there is a discriminant-preserving bijection

%\begin{equation}\label{voight-theta-quat}
%	\Set{\begin{array}{c}
%%			\text{quaternion orders over $\calo_F$}
%			\\
%			\text{up to isomorphism}
%	\end{array}}
%	\longleftrightarrow
%	\Set{
%		\begin{array}{c}
%			\text{nondegenerate ternary quadratic forms}\\
%			\text{over $F$, up to similitude}
%	\end{array}},
%\end{equation}
%which to any quaternion order $\calo$ associates the following ternary quadratic form. Let $(\calo^{\#})^0$ denote the elements of trace zero in the dual order $\calo^\#$, and let $N$ be a non-zero element of $\calo_F$ generating the reduced discriminant of $\calo$. Then
% \[
% (\calo^{\#})^0\longrightarrow \calo_F,\;\;\alpha\mapsto N\cdot\nr(\alpha)
%\]
%is a nondegenerate ternary quadratic form.

%Let $Q_\calo$ be the ternary quadratic form associated to $\calo$ by $(\ref{voight-theta-quat})$. Then we have that $Q_\calo(\alpha)=\delta$ if and only if $N\cdot\nr(\alpha)=m$.
%\end{proof}
%
%Therefore we can write the representation numbers as
%\begin{equation}
%	r(Q,\delta)=\#\{ I\subset \calo_B : \nr(I)=(m\cdot y) \}.
%\end{equation}

%Note that if the idele $y$ is not rational, then the ideal generated by $y$ is not principal.

As already hinted in Section \ref{weilrepr}, we make clear that the cases of interest for the present work are the following:
\begin{itemize}
	
	\item $(V,Q)=(B^0,Q_s)$;
	
	\item $(V,Q)=(F,U)$;
	
\end{itemize}

where $U$ is the unary form $U(x)=x^2$ and $Q_s$ is the ternary form as defined in Section \ref{ternary}.

\subsubsection{Ternary Quadratic Forms and Eichler Orders}\label{ternary}

The constructions of the Eichler orders attached to geometric points $s$ in the special fiber of the Shimura curve as in Section \ref{Grosslattice} provide the $\mathit{Gross}$ $\mathit{lattice}$ 
\[
\Lambda_s:=(2R_s+ \calo_F)\cap B'^{0}
\]
and the ternary quadratic forms
\[
Q_s\colon (2R_s+ \calo_F)\cap B'^{0}%=\{ 2x-\Tr(x) : %x\in R_s \}
\longrightarrow F,
\]
defined by $Q_s(b)=\nr(b)$.
\\Notice that, since $B'$ is ramified at $\tau_1$, then $Q_s$ is positive definite.

Let us now denote by $\Lambda_s$ the domain of $Q_s$. Extending the scalars by $\A_{F,f}$, we obtain the adelization of $Q_s$ 
\[
\widehat{Q}_s\colon \Lambda_s\otimes \A_{F,f}\longrightarrow \A_{F,f}.
\]

For $\delta\in\A_{F,f}^\times$, we consider the following  representation numbers
\[
r(\widehat{Q}_s,\delta):=\sum_{\substack{x\in (R_s\otimes\widehat{\calo}_F)^\times\backslash B^{'0}\otimes\A_{F,f} \\ \widehat{Q}_s(x)=\delta}}\mathbf{1}_{{R}_s\otimes \widehat{\calo}_F}.
%\otimes\widehat{\calo}_F^\times}.
\]

Let $(V,Q)$ denote a quadratic space over $F$ with a symmetric, bilinear form $\beta$, and let $(e_i)_{i=1}^d$ be a $F$-basis for $V$, and consider the $d\times d$-matrix $M=(\beta(e_i,e_j))_{i,j=1}^d$. Then the discriminant of the quadratic form $Q$ over $F$ is defined as the coset $D_Q=%(-1)^{\frac{d(d-1)}{2}}
\det M\cdot F^{\times,2}$ in $F^\times/ F^{\times,2}$.

\begin{lem}\label{disc}
  Suppose that $2$ is inert in $F$. Let $\gotn$ be the level of $R_s$. The discriminant $D_{Q_s}$ of the quadratic lattice $\Lambda_s$ is equal to the  $4\gotn^2 v^2$. 
\end{lem}

\begin{proof}
	Let us begin with the case of a place $v$ of $F$ such that $v\neq \ell$. %Let $\gotp$ be the prime corresponding to the place $v$. 
	Let $\varpi_{v}$ be the uniformizer of $\calo_{F_v}$. Denote by $n$ the $v$-adic valuation of $\gotn$,%the level of the Eichler order $R_s$, 
	so that $R_{s,v}:=R_s\otimes \calo_{F_v}$ is a Eichler order of level $\varpi_v^n$. We thus have
	\[
	R_{s,v}=\begin{bmatrix}
		\calo_{F_v} & \calo_{F_v}
		\\
		\varpi_{v}^n\calo_{F_v} & \calo_{F_v}
	\end{bmatrix}
	\]
	and consequently
	\[
	\Lambda_{s,v}:=\Lambda_s\otimes \calo_{F_v}=\left\lbrace \begin{bmatrix}
		a & 2b
		\\
		2\varpi_{v}^nc & -a
	\end{bmatrix}: a,b,c\in\calo_{F_v} \right\rbrace .
	\]
	
	Hence, by computing the determinant, the local quadratic form $Q_{s,v}$ is
	\[
	(a,b,c)\mapsto -a^2-4\varpi_v^{2n}bc.
	\]
	This shows that the contributions to the determinant of $Q_s$ are
	\[
	\begin{cases}
		& \varpi_v^{2n},\;\;\;\text{if}\; v \ndiv 2\calo_F
		\\
		& 4\varpi_v^{2n},\;\;\text{otherwise.}
	\end{cases}
	\]
	
	Consider now the case $v=\ell$. If $v \neq 2\calo_F$, then $R_{s,v}$ is the unique maximal order of the unique quaternion algebra, with $\calo_{F_v}$-basis $(1,i,j,k)$ such that $i^2=-\varpi_v$, $j^2=-1$, $k^2=-1$ and $k=ij=-ji$. Thus $\Lambda_{s,v}$ has a basis $(2i,2j,2k)$ amd we obtain the diagonal form 
	\[
	Q_{s,v}(a,b,c)=4\varpi_v a^2+4b^2+4\varpi_v c^2
	\]
	whose determinant is $64\varpi_v^2$, so contributing by $\varpi_v^2$ to the discriminant of $Q_s$.
	\\On the other hand, if $v= 2\calo_F$, then by \cite[pp.145,177]{gr}, the local Eichler order is isomorphic to the unique maximal order and for $i^2=j^2=k^2=-1$ and $k=ij=-ji$ we have
	\[
	\Lambda_{s,v}=\{ ai+bj+ck: a\equiv b\equiv c \bmod v \}.
	\]
	Therefore we have
	\[
	Q_{s,v}(a,b,c)=-(3a^2+4ab+4ac+4b^2+4c^2),
	\]
	
	whose corresponding matrix has determinant $16$, i.e., it contributes $4\ell^2$.
\end{proof}

%Let  $r(Q_s,\delta)$ be the number of representations of $\delta\in\A_{F}^\times$ by $\widehat{Q}_s$. 
We introduce, following \cite{kohnen},  the Kohnen plus space $M^{+}_{\frac{3}{2}}(U)$ as the subspace of $M_k(U)$ whose forms $f$ have Whittaker--Fourier coefficients $a(f,\delta)=0$, for $\delta\in\A_{F}^\times$, unless $\delta_v$ is congruent to a square modulo $4\calo_{F_v}$ for almost all $v$, i.e., there exists $y_v\in\calo_{F_v}$ such that $\delta_v\equiv y_v^2\bmod 4\calo_{F_v}$ for almost all $v$.

\begin{lem}\label{theta}
	The Gross theta series
	
	\begin{equation}
		\uptheta_{Q_s}=\sum_{\beta\in \Lambda_s} q^{Q_s(\beta)}
	\end{equation}
	
	are Hilbert modular forms which lie in the Kohnen's plus space $M^+_{\frac{3}{2}}(\Gamma_0(4\gotn \ell))$.
\end{lem}

\begin{proof}
	By the $q$-expansion map (\ref{expansionprinciple})  we have that the theta series $\uptheta_{Q_s}$
	are Hilbert modular forms and by (\ref{half}) it follows than its weight is $3/2$.
	
	For $\delta_v\in \calo_{F_v}$, we notice that
	$\calo_{F_v}+\dfrac{\delta_v+\sqrt{-\delta_v}}{2}\calo_{F_v}$
	is an $\calo_{K_v}$-order if and only if $\delta_v\equiv 0,1\bmod 4\calo_{F_v}$.

	Let us show the condition under which it is an $\calo_{K_v}$-order. We require that $(\delta_v+\sqrt{-\delta_v}/2)^2\in\calo_{F_v}$, which is equivalent to $(\delta_v^2-\delta_v)/4\in\calo_{F_v}$.
	Hence the sufficient and necessary condition to be an order is that $\delta_v\equiv 0,1\bmod 4\calo_{F_v}$. %Dividing by $4$ and repeating the process if $4\mid \delta_v$, we can immediately reduce to the case $\delta\equiv 1\bmod 4$.
	Since $Q_{s,v}(\beta_v)=-\delta_v\equiv 0,1 \bmod 4\calo_{F_v}$ we have that the coefficients of
	$\uptheta_{Q_{s,v}}$, i.e., the $r(Q_{s,v},\delta_v)$'s
	are non-zero only if $-\delta_v\equiv 0,1 \bmod 4\calo_{F_v}$. By Hasse--Minkowski theorem we obtain the global result.
	\\Concerning the level, we proceed as in \cite[Lemma 4.4]{jk}. Given that that $\uptheta_{Q_s}\in M^+_{3/2}(\Gamma_0(4\gotm))$, for $\gotm$ the smallest ideal such that $4\gotm $ has even integral coefficients, it is enough to see that the inverse of the matrices appearing in the proof of Lemma (\ref{disc}) have even integral coefficients after multiplying by $4 v^n$, where $n$ is the $v$-adic valuation of $\gotn \ell$.
\end{proof}

\subsubsection{Automorphic Theta Series}

In order to obtain automorphic theta forms from Weil representations, we need to introduce the analogue of  self-dual\footnote{With respect to the Fourier transform.} functions on the lattice $V(F)\simeq F^n$, for a quadratic space $(V,Q)$ with $Q$ a $n$-ary form.

For $v$ a non-archimedean place of $F$, consider the functions
\[
\mathbf{1}_{v}:=\prod_{i=1}^{n}\mathbf{1}_{\calo_{F_v}}.%\otimes \mathbf{1}_{\calo_{F_v,N}^\times}
\]
%where $\mathbf{1}_{\calo_{F_v}}$ denotes the characteristic function of $\calo_{F_v}$.
We have  that $\widehat{\mathbf{1}}_v=\mathbf{1}_v$, i.e., the Fourier transform of $\mathbf{1}_v$ is itself.
%where $\calo_{F_v,N}^\times=\{\delta\in\calo_{F_v}^\times : \delta\equiv 1\bmod N \}$. 
For $v=\infty$ an archimedean place, we define $\mathbf{1}_\infty$ to be the Gaussian exponential $\exp(-\pi Q(x))$. Note that $\mathbf{1}_v\in \call^2(F_v^n)$. Any adelic Schwartz--Bruhat function $\phi\in\mathtt{S}(\A_F^n)$ is a (finite) linear combination of the  product $\prod_v\phi_v$  where $\phi_v\in\mathtt{S}(F^n_v)$ and $\phi_v=\mathbf{1}_v$ for almost all $v$ (see \cite[Rmk.4.3.1]{hanke2}).

In order to associate to each of these quadratic forms a function on $\widetilde{\GL}_2(\A_F)$ we consider the following linear functional on $\mathtt{S}(\A^n_F)$
\[
\Theta\colon\phi\mapsto \sum_{x\in F^n}\phi(x)
\]
called $\mathit{theta}$-distribution. Note that the series defining the $\Theta(\phi)$ converges and $\Theta$ spans the space of $\SL_2(F)$-invariant forms on $\mathtt{S}(\A^n_F)$. %see Prasad thm 8.2
After inducing up to $\widetilde{\GL}_2(\A_F)$, the action of the Weil representation $r_Q$ attached to the form $Q$ on the theta-distribution defines
\[
\vartheta_Q(g,h,\phi)=\sum_{x}r_Q(g)\phi(h^{-1}x)
\]
for $g\in\widetilde{\GL}_2(\A_F)$ and $h\in O(V)(\A_F)$. This function is called $\mathit{theta}$-$\mathit{kernel}$ and it defines an automorphic form on $\widetilde{\GL}_2(\A_F)$.
\\Since $(V,Q)$ is anisotropic over $F$, it follows that $[O(V)]$ is compact. We thus consider
\begin{equation}\label{thetalift}
	\int_{[O(V)]}\vartheta_Q(g,h,\phi)dh,
\end{equation}
for $dh$ the normalized measure such that $[O(V)]$ is of $dh$-volume $1$. After inducing up to $\widetilde{\GL}_2(\A_F)$ again, (\ref{thetalift}) defines an automorphic form $ I(g,\phi)$ on $[\widetilde{\GL}_2]$.  For more details, see \cite[Section 2.6]{gel} and \cite[Section 4.4]{hanke2}.

\subsubsection{Ternary Quadratic Forms and Optimal Embeddings}

In this section we introduce one of the key ingredient of the equidistribution we are going to prove.

\begin{lem}\label{3.5}
	Let $\delta_{c}$ be the idele corresponding to the ideal $Dc^2$ by (\ref{idele-ideal}). Then we have the following $((\# R_s^\times/\# \calo^\times_{D,c}):1)$-correspondence 
	
	\[
	\Set{\begin{array}{c}
			\text{primitive representations}\\
			\text{$\widehat{Q}_s(b)=\delta_{c}$}  \\
	\end{array}}
	\longleftrightarrow
	\Set{
		\begin{array}{c}
			\text{optimal embeddings}
			\\
			\text{ $f\colon\calo_{D,c}\hookrightarrow R_s$ }
	\end{array}}.
	\]
	
\end{lem}
\begin{proof}
	First we show how to obtain a representation from an embedding and viceversa, inspired by \cite[Prop.12.9]{gr}.
	Indeed, we proceed locally at the finite places $v$ of $F$. Let $\delta_{c,v}=D_vc_v$ be the product of the generators of the ideals $D$ and $c$ at $v$. By the local-global principle for embeddings (see \cite[Prop.14.6.7]{voi}), we can write
	\[
	\calo_{D,c,v}=\calo_{F_v}+\dfrac{\delta_{c,v}+\sqrt{-\delta_{c,v}}}{2}.
	\]
	Consider an embedding $f_v\colon\calo_{D,c,v}\hookrightarrow R_{s,v}$ and denote by $\beta_v:=f_v(\sqrt{-\delta_{c,v}})$. Since 
	\[
	2f_v\left( \frac{1}{2}(\delta_{c,v}+\sqrt{-\delta_{c,v}})\right) =\delta_{c,v}+\beta_v,
	\]
	then $\beta_v\equiv -\delta_{c,v} \bmod 2R_{s,v}$, so that $\beta_v\in \Lambda_{s,v}$. Since $\nr_v(\beta_v)=\delta_v$, we obtained the desired representation.
	
	On the other hand, let $\beta_v\in \Lambda_{s,v}$ of reduced norm $\delta_{c,v}$. We thus write
	$\beta_v=\gamma_v+2r_v$, for $\gamma_v\in\calo_{F_v}$ and $r_v\in R_{s,v}$, so that
	\[
	\delta_{c,v}=\nr_{v}(b_v)=-(\gamma_v+2r_v)^2\equiv -\gamma_v\bmod 4R_{s,v}.
	\]	
	Thus we have $\beta_v\equiv -\delta_{c,v}\bmod 2R_{s,v}$, and we obtain an embedding $f_v\colon\calo_{D,c,v}\hookrightarrow R_{s,v}$ by setting
	\[
	f_v\left(\dfrac{\delta_{c,v}+\sqrt{-\delta_{c,v}}}{2} \right) =\dfrac{\delta_{c,v}+\beta_v}{2}.
	\]
	Finally, we recall that being  optimal is a local property, i.e., an embedding $f\colon \calo_{D,c}\hookrightarrow R_{s}$ is optimal if and only if $f_v\colon\calo_{D,c,v}\hookrightarrow R_{s,v}$ is optimal at all the places $v$ (see, for instance, \cite[Lemma 30.3.6]{voi}). Therefore, we proceed by localizing as above.
	
	Since we closely follow \cite[Lemma 4.1]{jk}, we only sketch the strategy to show the correspondence between optimal embeddings and primitive representations without the necessary algebraic manipulations. By contrapositive, if $Q_s(\beta_v)=\delta_{c,v}$ is not primitive, then we can write $\beta_v=\alpha_v k_v$ for $\alpha_v\in \Lambda_{s,v}$ and $k_v\in\calo_{F_v}$. By showing that $\frac{\delta_{c,v}+\alpha k_v^2}{2k_v^2}$ belongs to $f_v(K_v)\cap R_{s,v}$ but not to $f_v(\calo_{D,c,v})$, one concludes that $f_v$ is not optimal. 
	Lastly, if we suppose that $f_v$ is not an optimal embedding, then $f_v(\calo_{D,c,v}\otimes F_v)\cap R_{s,v}=\calo_{D,c',v}$ where $c=c'k$ for $k\in\calo_{F_v}$.	
\end{proof}

\subsubsection{Genus and Spinor Genus}

We revise here a few classical concepts following \cite{hanke2}. %An equivalent formulation can be found, for instance, in \cite[p.105]{jon}. 

Let $(V,Q)$ denote a quadratic space over $F$ with a symmetric, bilinear form $\beta$. We recall that another form $Q'$   is $\mathit{equivalent}$ to $Q$ if there exists an invertible linear change of variables $f$  such that $Q'(f(x))=Q(x)$. The set of $\calo_F$-equivalence classes that are $\calo_{F_v}$-equivalent locally for all $v$'s is the $\mathit{class}$ number of $Q$. The $\mathit{genus}$ of $Q$ is the set of all forms with the same localization as $Q$. The $\mathit{class}$ number of $Q$ is equal to the cardinality of $\gen(Q)$, which is finite by a classical result of Siegel. In terms of lattices, let $L$ be a integral lattice in $V$. Then we denote by $gen(L)$ the genus of $L$, defined to be the set of integral lattices $M$ such that 
$L_v\simeq M_v$
for every place $v$ of $F$ (including the archimedean ones). More geometrically, the genus of $L$ is the orbit of $L$ under the action of the %(rational or adelic)
%see Hanke p.21 sec 1.9
 orthogonal group (see \cite[Lemma 4.5.2]{hanke2}), which we denote by $O(V)$. Note that lattices in the same genus are isomorphic as $\calo_F$-modules. %so that the set of lattices is a disjoint union of genera
\\

In the case of ternary quadratic forms, the genus is moreover subdivived into $\mathit{spinor}$ genera, which essentially are the subgenera given by equivalence under the spinor group. More precisely, the $\mathit{spinor}$ norm is the group homomorphism
%see Hanke p. 147 for a nice explanation 
\begin{align*}
	&O(V)\rightarrow F^\times/ F^{\times,2}
	\\
	&\left( x\mapsto x-2u\beta(x,u)/Q(u) \right) \mapsto Q(u)
\end{align*}
i.e., it sends a reflection orthogonal to $v$ into $Q(v)$. We denote by $\Spin(V)$ the isometries of spinor norm $1$. This gives
 a two-fold cover of the special orthogonal group.
\\Two lattices $L$ and $M$ are in the same $\mathit{spinor}$ $\mathit{genus}$ if there is a rotation $\sigma\in \SO(V)$ and for every place $v$ there is $r_v\in \Spin(V_v)$ such that
$L_v\simeq \sigma_v r_v M_v.$
\\

Lastly, we recall that the $\mathit{automorphs}$ of $Q$ are the isometries from $Q$ to itself.  Let $w_Q$ be the number of the automorphs of $Q$, i.e., the cardinality of $\Stab_{\GL_3(\calo_{F})}(Q)$. 	Indeed the automorphs of $Q$ are finite: since  any of them is determined by its action on a  $\calo_F^3$-basis, and for $m$ big enough we have that the finitely many vectors in the compact defined by $Q(u)\le m$ generate $\calo_F^3$.

\subsubsection{Theta Series associated to Genus and Spinor Genus}\label{thetaseriesgenus}
We now associate two theta series to the genus $\gen(Q_s)$ and spinor genus $\spn(Q_s)$.

Let us define the genus and spinor genus $\mathit{mass}$\footnote{From the German word ``ma\ss'', which means ``weight, measure''.}
\begin{equation}\label{genmass}
	r^{?}(\diamondsuit, \delta)=\dfrac{\sum_{Q\in\diamondsuit}r^{?}(Q,\delta)/w_Q}{\sum_{Q\in\diamondsuit}1/w_Q}
\end{equation}

and
\begin{equation}
	\uptheta_\diamondsuit=\sum_{\delta\in\A_{F}}r(\diamondsuit,\delta)q^\delta
\end{equation}

for $\diamondsuit\in\{\gen(Q_s),\spn(Q_s) \}$ and $?\in\{\emptyset, *\}$, where $*$ denotes the restriction to primitive representations\footnote{And $\emptyset$ symbolizes the absence of any index.}.

\begin{prop}\label{3.10}
Let $s$ be in the supersingular or superspecial locus of the special fiber of $\cals_U$ at $v$, for $v$  unramified or ramified in $K$ respectively. Then we have:
	\begin{enumerate}
		\item the theta series $\uptheta_\diamondsuit$'s are in the same space as $\uptheta_{Q_s}$;
		\item the modular form
		\[
		\uptheta_{Q_s}-\uptheta_{\spn(Q_s)}
		\]
		of weight $3/2$  lies in the orthogonal\footnote{Where orthogonality is considered with respect to the Petersson inner product.} complement of the space of $1$-dimensional theta series;
		\item the modular form 
		\[
		\uptheta_{\gen(Q_s)}-\uptheta_{\spn(Q_s)}
		\]
		of weight $3/2$ lies in the space spanned by $1$-dimensional theta series.
	\end{enumerate}
\end{prop}

\begin{proof}

	Let us first  recall the very well known fact that the space of Hilbert modular forms  decomposes uniquely as a direct sum of cusp forms and Eisenstein series. In particular, the Hilbert modular form $\uptheta_{Q_s}$ defined in Lemma \ref{theta} decomposes as
	\[
	\uptheta_{Q_s}=E+H+s
	\]
	namely, the sum of an Eisenstein series $E$ and two cuspidal forms $H$ and $s$, where the Shimura lift of $H$ is an Eisenstein series, hence it is non-cuspidal.
	\\In the automorphic setting, we have the automorphic theta form $\vartheta_{Q_s}$ and, by Lemma \ref{dictionary}, $H$ and $s$ corresponds to two cuspidal automorphic forms which we still denote by  $\pi_H$ and $\pi_s$.  By Lemma \ref{3.17} and Proposition \ref{spectraldec} we can decompose the automorphic theta form as 
	\[
	\vartheta_{Q_s}=\Psi + \pi_H +\pi_s
	\]
	for a pseudo-Eisenstein series $\Psi$ and for $\pi_H\in \call_{\text{res}}$. By Section \ref{automorphicShimura}, we have that the ``kernel" of the Shimura lifting $\shi$ consists of those cuspidal $\overline{\pi}$'s which come from Hecke characters. This immediately implies that $\pi_H$ has a non-cuspidal Shimura lift.

Consider now the Whittaker--Fourier coefficients $W_E$ and $W_H$ of $E$ and $H$. By the work of Kudla and Rallis \cite{kudla} and \cite{kudla2}, we have the Siegel--Weil formula
\begin{equation}\label{siegelweil}
	I(g,\phi)=E(f,g)
\end{equation}
where $I(g,\phi)$ is the theta lift defined by (\ref{thetalift}), for $(V,Q)=(B^0,Q_s)$. The Siegel--Weil formula allows us to interpret an Eisenstein series as a weighted average of theta series of lattices in a genus. More precisely, since the orbit of  $Q_s$ under $O(B^0)(\A_F)$ corresponds to $\gen(Q_s)$, it follows that $O(B^0)(F)\backslash O(B^0)(\A_F)/\Stab(Q_s)$ is in bijection with the classes of $\gen(Q_s)$. Let us denote by $g$ the number of classes of such a genus. We decompose the adelic orthogonal group of $B^0$ as a disjoint union of double cosets corresponding to the classes of $\gen(Q_s)$ with respect to the stabilizer, i.e., $O(B^0)(\A_F)=\bigsqcup_{i=1}^gO(B^0)(F)\alpha_i\Stab(Q_s)$, for some representatives $\alpha_i\in O(B^0)(\A_F)$.
\\We thus obtain
\[
I(g,\phi)=\sum_{i=1}^g\dfrac{1}{w_{Q_i}}\int_{\Stab(Q_i)}\vartheta_{Q_s}(g,h\alpha_i)dh.
\]

By combining the previous formula with \cite[Satz 2, Satz 3]{sp} (see also \cite[p.366]{hanke}) we obtain
\begin{equation}\label{SP}
	W_E(\delta)=r(\gen(Q_s),\delta),\;\;W_E(\delta)+W_H(\delta)=r(\spn(Q_s),\delta).
\end{equation}
	By (\ref{SP}) and the fact that the $\uptheta_Q$'s are all in the same Kohnen space as $\uptheta_{Q_s}$ for $Q\in\gen(Q_s)$  we obtain that $\uptheta_\diamondsuit$'s are in the same space as $\uptheta_{Q_s}$. 
	
	We now prove the second part of the Lemma. In the automorphic setting, the comparison of the Whittaker--Fourier coefficients of the automorphic forms $\vartheta_{Q_s}$ and  $\vartheta_{\spn(Q_s)}$ show that their difference is not spanned by any form in $r(\chi)$. In modular forms terms, this means this difference lies in the orthogonal complement.
	
	The third part of the Lemma follows from analogous comparison of Whittaker--Fourier coefficients (in the modular form case due to \cite{sp}) of the automorphic forms associated to $\uptheta_{\gen(Q_s)}$ and $\uptheta_{\spn(Q_s)}$ and by Section $\ref{automorphicShimura}$.
\end{proof}

\subsection{Equidistribution}

We now state our main equidistribution result, whose proof is given in Section \ref{proofmainthm2}. From now on, let us denote by $w(s):=\# R_s^\times$, for $s$ a supersingular or a superspecial point. Similarly, let us denote by $w(c)$ the cardinality of the endormorphism ring of the reduction of a CM point landing on the component $c$ of the smooth locus $\cals_{U,k}^{\text{sm}}$.

\begin{thm}\label{mainthm2}
Let $\diamond\in\{\text{ss},\text{ssp}\}$. The sequence of measure $(\mu_{D,c}^\diamond)_{D,c}$ converges, in the weak-* topology, to the canonical measure
\[
\mu^\diamond(s):=\dfrac{w(s)^{-1}}{\sum_{s\in\cals_k^\diamond}w(s)^{-1}}
\]
as $D$ and $c$ vary, that is, as their absolute norms tend to infinity.

Moreover, we also have that the sequence $(\mu_{\text{in},D,c}^{\text{ssp}})_{D,c}$ converges to the measure
\[
\mu_{\text{in}}^{\text{ssp}}(c_i):=\dfrac{w(c_i)^{-1}}{\sum_{j=1}^h w(c_j)^{-1}}
\]               
on the set $\{c_1,\dots,c_h\}$  of components of the smooth locus $\cals_{U,k}^{\text{sm}}$.    
\end{thm}

\subsubsection{Bounds for the Whittaker--Fourier Coefficients}

The following estimate  makes essential use of the Brauer--Siegel theorem, i.e., a lower bound for the class number in the (totally real) number field setting. Note that the ineffectivity of our equidistribution results originates from the use of such notoriously ineffective bound.

\begin{lem}\label{subcon}
	Let $\Gamma_{D,c}$ be the Galois orbit of the CM point defined over the CM field $K$ of discriminant $Dc^2$.
	We have that
	\[
	\dfrac{r^*(Q_s,Dc^2)\cdot \#\calo_{D,c}^\times}{\#\Gamma_{D,c}}=\dfrac{r^*(\spn(Q_s),Dc^2)\cdot\#\calo_{D,c}^\times}{R_K\cdot\#\Gamma_{D,c}}+O(N(Dc^2)^{-\frac{7}{32}+\epsilon}).
	\]
\end{lem}
\begin{proof}
   Consider an irreducible cuspidal automorphic representation $\overline{\pi}$ of $\widetilde{\GL}_2(\A_F)$ orthogonal to the subspace spanned by theta-representations and pick an automorphic form $\overline{\upvarphi}$ of $\overline{\pi}$. %Let $r\in\calo_F$ be a non-zero square-free integer.
   %check the normalization used in BH p.31-32 (and Baruch Mao p.12) 
   By \cite[Cor.1]{bh} the following bound holds
   \begin{equation}\label{bhcor1}
   a(\overline{\upvarphi},m)\ll_{\overline{\upvarphi},F,\epsilon}N(m)^{\frac{1}{2}-\frac{1}{8}(1-2\theta)},
   \end{equation}
   where $a(\overline{\upvarphi},m)$ is the $m$-th normalized Whittaker--Fourier coefficient of $\overline{\upvarphi}$ and $\theta\in[0,\frac{1}{2}]$ is an approximation towards the Ramanujan--Petersson conjecture.
   By the work of Kim--Shahidi \cite{ks}, we can take $\theta=1/9$. Thus we obtain the exponent $\frac{1}{2}-\frac{1}{8}(1-\frac{2}{9})=\frac{29}{72}$.
   
   Now, by Proposition \ref{3.10}, we have that $\uptheta_{\spn(Q_s)}-\uptheta_{Q_s}$ is a cusp form of weight $3/2$ which lies in the orthogonal complement of the space of unary theta series. Consider the  automorphic representation of half-integral weight spanned by this form, and denote it by $\overline{\vartheta}$. Note that, since $\overline{\vartheta}$ it is not in the image of the correspondence of Lemma \ref{thetarepr}, its Shimura lift $\shi(\overline{\vartheta})$ is cuspidal. Thus by (\ref{bhcor1}) it follows that
   \[
   r^*(\spn(Q_s),Dc^2)-r^*(Q_s,Dc^2)\ll_{\overline{\vartheta},F,\epsilon} |N(Dc^2)|^{\frac{29}{72}+\epsilon}.
   \]
   Furthermore by the Brauer--Siegel theorem (see, for instance, \cite[Chap.XVI]{lang}) we have the following lower bound for the class number $h_K$
   \[
   h_K R_K\gg_{\epsilon} N(Dc^2)^{\frac{1}{2}-\epsilon}
   \]
   where $R_K$ is the regulator of $K$. Since $K$ is a CM field, by \cite[Thm.B]{fried} we have that $R_K>1/4$. Therefore
   \[
  	\dfrac{r^*(Q_s,Dc^2)\cdot\#\calo_{D,c}^\times}{\#\Gamma_{D,c}}-\dfrac{r^*(\spn(Q_s),Dc^2)\cdot \#\calo_{D,c}^\times}{\#\Gamma_{D,c}\cdot R_K}\ll_{}N(Dc^2)^{-\frac{7}{32}+\epsilon}.
   \]
\end{proof}

For our equidistribution purposes, it is important the trivial remark that the exponent $\frac{29}{72}$ is smaller that the Duke's one $\frac{13}{28}$ as considered in \cite[p.518]{jk}.

\subsubsection{Auxiliary Results.}

Let $\left(\dfrac{K/F}{\cdot} \right)$ denote the Artin symbol.

\begin{lem}\label{classnumber}
	Let $h(\calo_c)$ and $h(\calo_K)$ denote the class numbers of $\calo_c$ and $\calo_K$ respectively. Then we have
	\[
	h(\calo_c)=h(\calo_K)\dfrac{N(c)}{[\calo_K^\times:\calo_c^\times]}\prod_{\gotp|c} \left( 1-\left( \dfrac{K/F}{\gotp}\right) \dfrac{1}{N(\gotp)}\right).
	\]	
\end{lem}
\begin{proof}
	We follow the lines of \cite[Thm.7.24]{cox}. Let $I_K(c)$ denote the group of fractional $\calo_K$-ideals prime to $c$  and  let $P_{K,\calo_F}(c)$ be the subgroup of $I_K(c)$ generated by the principal ideals $\alpha\calo_K$, for $\alpha\in\calo_K$ such that $\alpha\equiv a\bmod c\calo_K$ for $a\in\calo_F$ coprime to $c$. Then, by an immediate adaptation of \cite[Prop.7.22]{cox}, we have 
	\[
	h(\calo_c)=\#\dfrac{I_K(c)}{P_{K,\calo_F}(c)}.
	\]
	The inclusion $P_{K,\calo_F}(c)\subset I_K(c)\cap P_K$ allows us to consider the short exact sequence
	\[
	\begin{tikzcd}
	0 \arrow{r} & I_K(c)\cap P_K/ P_{K,\calo_F}(c)\arrow{r} &\Cl(\calo_c)\arrow{r} &\Cl(\calo_K)\arrow{r} & 0,
	\end{tikzcd}
	\]
	so that we just need to compute
	\begin{equation}\label{3.12}
	\#(I_K(f)\cap P_K/ P_{K,\calo_F}(f))=h(\calo_c)/h(\calo_K).
	\end{equation}
	For $[\alpha]\in(\calo_K/ c\calo_K)^\times$ let us consider the surjective morphism
	\[
	\phi\colon (\calo_K/c\calo_K)^\times\longrightarrow I_K(c)\cap P_K/P_{K,\calo_F}(c)
	\]
	defined by $[\alpha]\mapsto [\alpha\calo_K]$.
	For the sake of simplicity, let us now assume that $\#\calo_K^\times=2$. We thus obtain the following short exact sequence
\[	
\begin{tikzcd}
		1 \arrow{r} &(\calo_F/c)^\times \arrow{r}{i} &(\calo_K/c\calo_K)^\times \arrow{r}{\phi} & I_K(c)\cap P_K/P_{K,\calo_F}(c)\arrow{r} & 1
\end{tikzcd}
\]
where $i$ is the natural injection. Analogously to the classical case, the number of units in the quotient $\calo_F/c$ is given by 	
\begin{equation}\label{wn}
\#(\calo_{F}/c)^\times= N(c)\prod_{\gotp|c}\left( 1-\dfrac{1}{N(\gotp)}\right) 
\end{equation}
(see \cite[Thm.1.19]{wn}). Combining formula (\ref{wn}) with the Chinese remainder theorem we thus obtain
\begin{equation}\label{3.13}
\#(\calo_K/ c\calo_K)^\times= N(c)^2
\prod_{\gotp|c}\left(1-\dfrac{1}{N(\gotp)}\right) \left( 1-\left( \dfrac{K/F}{\gotp}\right) \dfrac{1}{N(\gotp)} \right) 
\end{equation}
where indeed $N(c)^2=N(c\calo_K)$. Note that the Artin symbol summarizes the sign coming from the primes of $c$ being split or inert in $K$.

Finally, combining (\ref{3.13}) with (\ref{wn}) and (\ref{3.13}) we obtain
\[
h(\calo_c)/h(\calo_K)=N(c)\prod_{\gotp|c}\left(1-\left(\dfrac{K/F}{\gotp}\right) \dfrac{1}{N(\gotp)} \right).
\] 
Lastly, if $\#\calo_c^\times>2$, it is enough to consider the following exact sequence
\[
\begin{tikzcd}
	1\arrow{r}&\{\pm 1\}\arrow{r}{j}&(\calo_F/c)^\times \times \calo_K^\times \arrow{r}{\psi}& (\calo_K/c\calo_K)^\times\arrow{r}{\phi}&I_K(c)\cap P_K/P_{K,\calo_{F}}(c)\arrow{r}&1
\end{tikzcd}
\]
where $j\colon \pm1\mapsto ([\pm1],\pm 1)$ and $\psi\colon([n],u)\mapsto[nu^{-1}]$.
\\In this case, we have 
\[
\#(I_K(c)\cap P_K/ P_{K,\calo_F}(c))=2\dfrac{\#(\calo_K^\times/c\calo_K)^\times}{\#(\calo_F/c)^\times\#\calo_K^\times}.
\]
Again, by (\ref{3.12}), we conclude
\[
h(\calo_c)/h(\calo_K)=\dfrac{N(c)}{[\calo_K^\times:\calo_c]}\prod_{\gotp|c}\left(1-\left(\dfrac{K/F}{
\gotp}\right) \dfrac{1}{N(\gotp)}\right).
\]
\end{proof}

From now on, let us set $u_{D,c}:=\#\calo_{D,c}^\times$. We write $u_{D,1}$ to denote the cardinality of the order of trivial conductor.

	%By our assumption that $2$ is inert in $F$, this follows from \cite[Thm. 86]{jon} and from the local consideration in \cite[p.517]{jk}. %Namely by the computations in Lemma \ref{disc}, can be applied almost verbatim. 
%	\\Namely, as \cite[Thm.86]{jon}, this follows from a combination of \cite[Thms.49, 72, 85]{jon}. In view of this I just give a dictionary to help recovering this result from Jones' and reduce notational confusion. In particular,  we have that $G(A,q)$ corresponds to $r^*(\gen(Q_s),\delta)$, and in his notation $d$ corresponds to our discriminant $D_{Q_s}$. Moreover, as in \cite[p.517]{jk}, by analogue local calculations one can show that the discriminant $D_{Q_s}$ is equal to the square of the greatest common divisor of the determinants of all $2\times 2$ minors of the matrix corresponding to $Q_s$, denoted by $\Omega^2$. Thus 

\begin{lem}\label{limitex}
	We have that the limit defined as
	\[
	\lim_{k\to\infty}\dfrac{r^*(\gen(Q_s), D\gotp^{2k})u_{D,\gotp^k}}{\# \Gamma_{D,\gotp^k}}
	\]
	exists and it is independent of $s$ and $\gotp$.
\end{lem}

\begin{proof}
	The independence of $s$ follows from the definition of $\gen(Q_s)$.
	
	By Lemma \ref{3.5} the number of optimal embeddings of $\calo_{D,\gotp^k}$ into $R_s$ is given by $r^*(Q_s,D\gotp^{2k})u_{D,\gotp^k}w_s^{-1}$, and taking the sum of the geometric points $s$'s over the supersingular or superspecial locus gives $A\cdot h(D\gotp^{2k})$, where the constant $A$ depends only on the Artin symbol.
	By the Siegel--Weil formula (as in (\ref{siegelweil}) ) it follows that
	\begin{equation}\label{jon}
	r^*(\gen(Q_s),D\gotp^{2k})=A\dfrac{h(D\gotp^{2k})}{u_{D,\gotp^{k}}}.
	\end{equation}

	On the other hand, by  combining equation (\ref{jon}) with Lemma \ref{classnumber} we have
\begin{align}\label{3.17}
	& r^*(\gen(Q_s),D\gotp^{2k})\sim N(\gotp^k)\left(1-\dfrac{1}{N(\gotp)}\left( \dfrac{K/F}{\gotp}\right)  \right)\dfrac{h(\calo_K)}{u_{D,1}}  
	\\
	&\#\Gamma_{D,\gotp^k}=N(\gotp^k)\left(1-\dfrac{1}{N(\gotp)}\left( \dfrac{K/F}{\gotp}\right)\right)  \dfrac{u_{D,\gotp^k}}{u_{D,1}}\#\Gamma_{D,1}
\end{align}
where $\sim$ means equality up to a constant $C$ depending on $D$ and on the Artin symbol only.

	Therefore we obtain
	\[
	\dfrac{r^*(\gen(Q_s), D\gotp^{2k})u_{D,\gotp^k}}{\# \Gamma_{D,\gotp^k}}=C\dfrac{h(D)}{\#\Gamma_{D,1}}
	\]
	where the right-hand side is independent of $\gotp$.
\end{proof}

\begin{lem}\label{vat}
	Let $c$ be coprime to $\gotn$ and $\ell$. 
Then, for every $s$ supersingular or superspecial, in the special fiber at an inert or ramified prime respectively, we have that
	\begin{equation}
	f_{D,\gotp}(s):=\lim_{k\to\infty}\dfrac{r^*(Q_s,D\gotp^{2k})u_{D,\gotp^k}}{\#\Gamma_{D,\gotp^k}}=\# R_s^\times\cdot \mu^{\diamond},
	\end{equation}
where $\diamond\in\{\text{ss},\text{ssp}\}$.
\end{lem}
\begin{proof}
By Section \ref{liftred}, we have maps to pass from Gross points to supersingular or superspecial points. In view of those constructions, Lemma \ref{vat} in the supersingular case follows almost verbatim from \cite[Thm.1.5]{vat} and from Lemma \ref{3.5}. On the other hand, the superspecial setting hardly changes the proof. In fact, note that the auxiliary prime $\ell$, which one suppose unramified in $B$, is different from the prime of reduction $v$, which thus can be of ramification in $B$. In both cases, the equidistribution is reduced to a classical statement on finite graphs \cite[Prop.3.14]{vat}.
\end{proof}

\begin{rmk}
We invite the reader to note how Vatsal's proof of the previous Lemma is crucial for our results: the absence of any ergodic theory is the reason why we can allow the discriminant to vary in Theorem \ref{mainthm2}, in contrast\footnote{There is a price to pay: Vatsal's equidistribution holds for the reduction at a single prime only.} to \cite{cv} and \cite{fms}. 
\end{rmk}

Next result shows that there is only a single spinor genus.

\begin{lem}\label{gen=spn}
Let $s$ be either in the supersingular or in the superspecial locus of $\cals_U$. Suppose also that $c$ is coprime to $\gotn$ and $\ell$. Then
\[
r^*(\gen(Q_s),Dc^2)=r^*(\spn(Q_s),Dc^2).
\]
\end{lem}
\begin{proof}
We proceed as in \cite[p.520]{jk}, i.e., by reductio ad absurdum. Let $E$ be the extension of $F$ of discriminant $DD_{Q_s}$. Consider a prime $\gotp$ coprime to $N$ and $\ell$ such that  $\left(\dfrac{E/F}{\gotp} \right) =-1$.
Denote by $W(m)$ the $m$-th Whittaker--Fourier coefficient of $\uptheta_{\gen(Q_s)}-\uptheta_{\spn(Q_s)}$.

By (\ref{3.17}) we have
	\begin{align*}
		f_{D,\gotp}(s)&=
	    \lim_{k\to\infty}\dfrac{(r^*(\spn(Q_s),D\gotp^{2k})-r^*(\gen(Q_s),D\gotp^{2k}))u_{D,\gotp^k}}{\#\Gamma_{D,\gotp^{k}}}+\dfrac{r^*(\gen(Q_s),D\gotp^{2k})u_{D,\gotp^k}}{\#\Gamma_{D,\gotp^k}}
	    \\
	    &=\lim_{k\to\infty}\dfrac{W(D\gotp^{2k})}{\#\Gamma_{D,\gotp^{k}}u_{D,\gotp^k}^{-1}} + \dfrac{r^*(\gen(Q_s),D\gotp^{2k})u_{D,\gotp^k}}{\#\Gamma_{D,\gotp^k}}.
	\end{align*}
 By  Proposition \ref{3.10} $W(D\gotp^{2k})$ is of weight $3/2$ and spanned by $1$-dimensional theta series as implied by Lemma \ref{thetarepr}. By \cite{shi} and following \cite[Cor.3.5.3]{sz2}, we can write, in analogy with Remark \ref{Unary}, such a Hilbert theta series as
 \[
 \uptheta_\omega(z)=\sum_m \omega(m)N(m)q^{N(m^2)}
 \] 
 where $\omega$ is a Hecke character and $m$ an integral ideal.
 Therefore one obtains the formula
\begin{equation}\label{lemma4.9}
	W(D\gotp^{2k})=N(\gotp)^{k}\left( \dfrac{E/F}{\gotp}\right)^k W(D) 
\end{equation}
by the analogous steps of \cite[Lem.4.9]{jk}. Hence by formula (\ref{lemma4.9}) and formulae (\ref{3.17}) we can write
\[
N(\gotp)^k\left(\dfrac{E/F}{\gotp} \right)^k W(D) \cdot\dfrac{u_{D,1}}{N(\gotp)^k \left( 1- \dfrac{1}{N(\gotp)} \left( \dfrac{K/F}{\gotp}\right)\right)\#\Gamma_{D,1}}.
\]
By Lemma \ref{vat}, we have that the limit
\[
\lim_{k\to\infty}\dfrac{r^*(\gen(Q_s),D\gotp^{2k})u_{D,\gotp^k}}{\#\Gamma_{D,\gotp^{k}}}
\]
exists. However, the limit\footnote{Where the symbol $\sim$ means equality up to a constant.}
\[
\lim_{k\to\infty}\dfrac{W(D\gotp^{2k})}{\#\Gamma_{D,\gotp^{k}}u_{D,\gotp^k}^{-1}}\sim \lim_{k\to\infty}(-1)^k
\]
does not exist. This yields the desired contradiction.
\end{proof}

\subsubsection{Proof of Theorem \ref{mainthm2}.}\label{proofmainthm2}

Combining the subconvexity bound of Lemma \ref{subcon} with Lemma \ref{gen=spn}, we obtain
\[
\dfrac{r^*(Q_s,Dc^2)u_{D,c^2}}{\#\Gamma_{D,c}}=\dfrac{r^*(\gen(Q_s,Dc^2))u_{D,c}}{R_K\cdot\#\Gamma_{D,c}}+O(N(Dc^2)^{-\frac{7}{32}+\epsilon}).
\]
Let $s$ be a geometric point either in the supersingular or in the superspecial locus. By Lemma \ref{3.5} we have that the number of optimal embeddings of $\calo_{D,c}$ into $R_s$ is given by
\begin{equation}\label{3.20}
r^*(Q_s, Dc^2)\dfrac{u_{D,c}}{w_s}.
\end{equation}
Since the sum over all geometric points $s$ of (\ref{3.20}) is equal to $\#\Gamma_{D,c}$, we obtain
\begin{align*}
1&=\sum_{s\in\cals_k^\diamond} \dfrac{r^*(Q_s,Dc^2)u_{D,c}}{w_s\cdot\#\Gamma_{D,c}}
\\
&=\dfrac{r^*(\gen(Q_s,Dc^2))u_{D,c}}{R_K\cdot\#\Gamma_{D,c}} \sum_{s\in\cals_k^\diamond}\dfrac{1}{w_s} + O(N(Dc^2)^{-\frac{7}{32}+\epsilon}),
\end{align*}

where the independence of $r^*(\gen(Q_s,Dc^2))$ from $s$ gives the second equality. 

Since
\[
\dfrac{r^*(\gen(Q_s,Dc^2))u_{D,c}}{R_K\cdot\#\Gamma_{D,c}}=\dfrac{1}{\sum_{s\in\cals_k^\diamond}w_s^{-1}}+O(N(Dc^2)^{-\frac{7}{32}+\epsilon})
\]
we have that $\lim_{N(Dc^2)\to \infty}\dfrac{r^*(\gen(Q_s,Dc^2))u_{D,c}}{R_K\cdot\#\Gamma_{D,c}}$ exists. 
\\Therefore
\begin{align*}
	&\lim_{N(Dc^2)\to \infty}\dfrac{r^*(Q_s,Dc^2)u_{D,c}}{\#\Gamma_{D,c}}
	\\
	&=\lim_{N(Dc^2)\to \infty} \dfrac{r^*(\gen(Q_s,Dc^2))u_{D,c}}{R_K\cdot\#\Gamma_{D,c}}+O(N(Dc^2)^{-\frac{7}{32}+\epsilon})
	\\
	&=\dfrac{1}{\sum_{s\in\cals_k^\diamond}w_s^{-1}}=w_s\mu^{\diamond}
\end{align*}

and so we conclude.

\section{An Andr\'e--Oort-like Result}\label{sectionAO}

Following \cite{aoexp} and \cite{ao}, we consider the $2$-dimensional arithmetic Andr\'e--Oort conjecture \cite[Conj.2.3]{ao} for the integral model $\cals_U$ over $\Z$ of the Shimura curve $\scal_U$ of level $U$ over $\Q$. In particular, we extend the main result of \cite{ao}, which holds for $Y(1)$ exclusively, to the general case of quaternionic Shimura curves. Despite the fact that such variation of the Andr\'e--Oort conjecture is stated in \cite{aoexp} for integral models over the ring of integers of (a finite extension of) the field of definition of the Shimura variety, to better highlight the equidistribution argument, by avoiding extensions of several analytic estimates  to the number field setting, we reduce to the rational case. Note that what follows holds, mutatis mutandis, also for a Shimura curve $\scal_U$ attached to $B=M_2(\Q)$, i.e., for modular curves with arbitrary level structure. For the general statement of this conjecture for Shimura varieties, see \cite[Section 2.1]{ao}. 

\subsection{Horizontal Andr\'e--Oort in Pencils}
Let $p$ be a rational prime. Intuitively, the word ``horizontal" refers to the fact that we allow the characteristic $p$ to vary.

Let $\scal(d_K)^{\text{CM}}$ denote the set of CM points with discriminant $d_K$. Note that two CM points have the same discriminant if and only if they belong to the same $G_\Q$-orbit. A $\mathit{special}$ Cartier divisor $S$ is an element of the group of divisors $\Div(\cals_{\overline{\F}_p})$ of the form
\[
\sum_{z\in\scal(d_K)^{\text{CM}}}[\text{red}_p(z)],
\]
for some $p$ and some $d_K$. Note that the degree of $S$ coincides with its cardinality $\#S$.

Consider the Galois orbit of CM points of discriminant $d_K$. Denote by  $E$  a field and let $s=\Spec \overline{E}$. For an $E$-scheme $X$ locally of finite type, we recall that there is an equivalence between scheme theoretic points of $X$ and Galois orbits of geometric points. Namely, the map
\[
X(\overline{E})\longrightarrow X,\;x\mapsto x(s)
\]
induces a bijection between the set of $G_E$-orbits in $X(\overline{E})$ and the set of closed  points in $X$. Therefore such a $G_\Q$-orbit corresponds to a unique closed point $x(d_K)$ in $\scal$ over $\Q$. Note that the Zariski closure in $\cals$ of $x(d_K)$, which we call $\mathit{special}$ curve, is such that $\overline{x(d_K)}^{\text{Zar}}(\overline{\F}_p)=\text{red}_p(\scal(d_K)^{\text{CM}})$. 
\\

This horizontal extension of the Andr\'e--Oort conjecture (as proposed in \cite[Conj.1]{aoexp}) concerns the Zariski closure of a collection of special divisors in $\cals$. 

\begin{thm}\label{hao}
	Let $\mathbb{S}$ denote a collection of special divisors inside $\cals$. Then the Zariski closure (in $\cals$) of $\cup\{S : S\in\mathbb{S}\}$ is a finite union of the following subsets:
	\begin{enumerate}
		\item a special divisor $S$, for $p$ and $d_K$ fixed;
		\item the special curve $\overline{x(d_K)}^{\text{Zar}}$, for $d_K$ fixed;
		\item the fibers $\cals_{\F_p}$ over $p$, for $p$ fixed;
		\item the integral model $\cals$.
	\end{enumerate}
\end{thm}

As in \cite[Def.3.1]{ao}, we call the $\mathit{characteristic}$ of a special divisor $S$ the prime $p$ such that $S\subset \cals_{\overline{\F}_p}$. 
The structure of a special divisor $S$ falls under one of the following cases:
\begin{enumerate}
	\item if $B_p$ is ramified, it consists 
	\begin{itemize}
		\item either of superspecial points;
		\item or of supersingular points;
	\end{itemize}
	\item if $B_p$ is unramified, it consists
	\begin{itemize}
		\item either of supersingular points;
		\item or of ordinary points.
	\end{itemize}
\end{enumerate}

We thus label a special divisor as $\mathit{superspecial}$, $\mathit{supersingular}$ or $\mathit{ordinary}$ if it lies either in the superspecial or in supersingular, or in the ordinary locus of the special fiber of $\cals$. Note that such a divisor will be entirely contained in only one of these sets.

\subsubsection{Counting false elliptic curves modulo  $p$}

We conclude with a result in the spirit of \cite{KM} and \cite{bz} for our Shimura curves  and that we are gonna exploit in the next final session.

As in \cite[p.364]{kass}, we recall that a $\calo_B\otimes\Z_p$-module $N$ decomposes as $N=N_1\oplus N_2=N_1\oplus(N_{2,1}\oplus N_{2,2})$, where the $M_2(\Z_p)$-module $N_2$ decomposes\footnote{After choosing an idempotent of $M_2(\Z_p)$.} into the sum of two $\Z_p$-modules.
\\Consider the following moduli problem $\calf$ of $\Z_p$-algebras
\[
R\longmapsto[A, \theta, \kappa]
\]
where the triple $[A,\theta,\kappa]$ is an isomorphism class consisting of
\begin{enumerate} 
	\item  an abelian surface $A$ over $R$ with an action $\iota\colon\calo_{B}\hookrightarrow\End_R(A)$ such that $\Lie(A)_{2,1}$ is of rank-$1$ and $\Z_p$ acts on it\footnote{Since $\Lie(A)$ is a $\calo_{B}\otimes\Z_p$-module, it admits the above decomposition.};
	\item  a class of polarizations $\theta\colon A\rightarrow A^\vee$ such that, for every $b\in\calo_B$, the associated Rosati involution takes $\iota(b)$ to $\iota(b^*)$;
	\item a class of $\calo_{B}$-linear rigidifications $\kappa\colon \calo_{B}\otimes\widehat{\Z} \simeq \prod_p T_p(A)$.
\end{enumerate}

This moduli problem is representable by $\cals\times\Spec\Z_p$. For more details, we refer to \cite[Section 4.1]{kass} and \cite[Prop.1.1.5]{sz2}.

Let us consider the category of QM abelian surfaces base-changed to $\overline{\F}_p$ and the moduli problem $\overline{\calf}$ on it defined by base-change over $\overline{\F}_p$.
Indeed, $\overline{\calf}$ is representable by the special fiber $\cals\times\Spec\overline{\F}_p$. By this moduli interpretation we can thus define supersingular and superspecial points as in Section \ref{reductionintegralmodels}.
\\

Fix an isomorphism 
\[
\varprojlim_{(N,\Delta)=1}\Z/N\otimes_\Z \calo_B\simeq M_2(\varprojlim_{(N,\Delta)=1}\Z/N).
\] 
This induces a map $l_N\colon\widehat{\calo}^\times_B\rightarrow\GL_2(\Z/N)$. Let $V_1(N)$ be the preimage under $l_N$ of
\[\left\lbrace 
\begin{bmatrix}
a & b\\
c & d
\end{bmatrix}\in\GL_2(\Z/N) : c=0,\;d=1
\right\rbrace.
\]
Let now $U$ be a compact open subgroup of $\widehat{\calo}^\times_B$ that is maximal at the primes $p$ dividing $\Delta$. We require that $U\subset V_1(N)$ for $N>3$, again prime to $\Delta$. This condition on $N$ is indeed the usual ``smallness" condition to avoid stacky situations. Let $N_U=\min\{M\in\Z_{\ge0} : (M,\Delta)=1,\; \ker(l_M)\subset U \}$. Therefore the Shimura curve $\cals_U$ is the moduli space of false elliptic curves over $\Spec\Z[1/N_U\Delta]$ with a level $U$-structure as defined in \cite[Definition 1.1]{buzzard}.
\\We finally denote by $g_U$  the genus of any geometric fiber of $\cals_U$ over $\Z[1/N_U\Delta]$, and we also denote by $h(\Delta,N)$ the class number of any Eichler order of level $N$ in a quaternion algebra over $\Q$ of (reduced) discriminant $\Delta$.

\begin{lem}\label{numberss}
	Let $p$ be a rational prime not dividing either $\Delta$ or $N_U$. The number of supersingular points in the special fiber of $\cals$ at $p$ is given by
	\[
	\#\cals_{U,\overline{\F}_p}^{\text{ss}}=(p-1)(g_U-1).
	\]
	On the other hand, if $p$ divides $\Delta$, then the number of superspecial points in the special fiber at $p$ is given by
	\[
	\#\cals^{\text{ssp}}_{\overline{\F}_p}=h\left(\dfrac{\Delta}{p},Np \right) .
	\]
	
\end{lem}
\begin{proof}
	Let us refresh a few facts on the Hasse invariant $H$ for Shimura curves, following \cite[Section 6]{kass} and \cite[Section 5]{buzzard}. Firstly, consider the universal abelian scheme $\varepsilon\colon\cala\rightarrow\cals$ as in Section \ref{universal}, and let $\Omega^1_{\cala/\cals}$ be the canonical bundle. Note that the $\calo_\cals$-module $\varepsilon_*\Omega^1_{\cala/\cals}$ is also a $\calo_{B_p}$-module, and we introduce $\underline{\omega}=(\varepsilon_*\Omega^1_{\cala/\cals})_{2,1}$, which is a line bundle\footnote{This follows from the fact that $\Lie(\cala)_{2,1}$ is locally free of rank-$1$.} over $\cals$.
	 
	We now have that $H$ is a modular form of weight $p-1$ over $\cals\otimes\overline{\F}_p$ constructed as a section of the bundle $\underline{\omega}^{\otimes(p-1)}$ over $\cals\times\Spec\overline{\F}_p$. By \cite[Prop.6.1,6.3]{kass} it follows that $H$ vanishes exactly in the supersingular locus and it has only simple zeros. Therefore the number of supersingular points, i.e., the number of zeros of $H$ counted with multiplicities, is equal to
	\[
	\deg\underline{\omega}^{\otimes(p-1)}=(p-1)(g_U-1).
	\]
	by the Riemann--Roch theorem (as $\Omega^1\simeq\underline{\omega}^{\otimes 2}$). 
	
	Lastly, let us consider the case of $p$ dividing $\Delta$, i.e., that of $B_p$ ramified. In this case, the set of superspecial points corresponds to the edges of the dual graph of the special fiber of the Shimura curve, as explained in \cite[Section 2.2.5]{fms}. The cardinality of such a set of edges is then given in \cite[p.9]{rotger}.
\end{proof}

\begin{rmk}
	Lemma \ref{numberss} can be made more explicit: in fact, a formula for the genus of a Shimura curve over $\Q$ of discriminant $\Delta$  and level $N$, for $N$  squarefree, can be found in \cite[p.280,301]{ogg}. %Note that, as $(p,\Delta N_U)=1$, such a formula does not involve the prime $p$.
	Moreover, see \cite[p.152]{vign} for an explicit formula for the class number of an Eichler order of level $N$.%indeed h(11,7)=8.
\end{rmk}

\subsubsection{Proof of Theorem \ref{hao}}

Armed with the constructions and results of the rest of this paper, we follow the line of reasoning of \cite{ao}. Note that, as for the classical Andr\'e-Oort conjecture, where the level structure is inessential (see, for instance, \cite{edix}), also in our case what follows holds  with arbitrary level structure. 
\\

We begin with  parts (1) and (3) of Theorem \ref{hao}. Since the fiber of the map $\cals_U\rightarrow \Spec\Z$ above a prime $p$ is the curve $\cals_U\times\Spec\F_p$, then one easily concludes, because the special fiber of $\cals_U$ at $p$ is of dimension one. It thus follows that there is either a finite set of special points, so the closure is the just their union, or  
%infinitely many special points, and
the closure of the (infinitely many) special points is $\cals_{U,\F_p}$ itself, because the closure of any infinite subset of an irreducible curve is the whole curve.
\\

Concerning the case (2) of Theorem \ref{hao}, we reduce to the case of special divisors of bounded degree and characteristic going to infinity. Moreover, we can also reduce to the case of  $\overline{\cup\{S : S\in\mathbb{S}\}}^{\text{Zar}}\rightarrow\Spec\Z$ sending its generic points to the generic point of $\Spec\Z$. To show this, let $Z'$ denote the union of the irreducible components of the Zariski closure of $\cup\{S : S\in\mathbb{S}\}$ which are  contained in a special fiber $\cals_{U,\F_p}$. Since the characteristic is not bounded, in $Z'$ there are a finite number of special divisor. We thus consider $Z$ to be $Z'$ minus such divisors. Thus $Z\rightarrow\Spec\Z$ sends the generic points of $\overline{\cup\{S : S\in\mathbb{S}\}}^{\text{Zar}}$ to the generic point of $\Spec\Z$.

Let $p$ be unramified in $B$ and split in $K$, i.e., the associated special divisor $S$ is ordinary. We say that $S$ is $\mathit{canonical}$ if it is the special fiber at $p$ of a special curve $\overline{x(d_K)}^{\text{Zar}}$ such that $\# \text{red}_p(\scal(d_K)^{\text{CM}})=\# \scal(d_K)^{\text{CM}}$. By Serre--Tate theory, there is a unique such canonical lifting $x(d_K)$. For a detailed account on this beautiful theory, we invite the reader to go through \cite{cco}.
\\As in \cite[Section 3.2.4.1]{ao}, there are only finitely many special curves of bounded degree. Moreover, the ordinary special divisors lift to special curves with bounded discriminant, so that we have only a finite number of discriminants and  we deal with a finite union of special curves. This case is $1$-dimensional, so we conclude as for parts (1) and (3).

Let $p$ be unramified in $B$ and inert  in $K$, i.e., the associated $S$ is supersingular.

As in \cite[Thm.3.8]{ao}, this case reduces to showing that
\begin{equation}\label{lim}
\lim_{p\to \infty}\lim_{d_K\to -\infty}\# \text{red}_p(\scal(d_K)^{\text{CM}})=\infty,
\end{equation}
for $p$ coprime to the conductor $c$.
We also need to impose that $p\to\infty$ for supersingular special divisor. In fact, for a fixed $p$,
\[
\lim_{d_K\to -\infty}\# \text{red}_p(\scal(d_K)^{\text{CM}})=\#\cals^{\text{ss}}_{U,\overline{\F}_p}=
(p-1)(g_U-1).
\]

where the first equality is a consequence of Theorem \ref{mainthm2}  the cardinality of the supersingular locus comes from Lemma \ref{numberss}.

In order to prove (\ref{lim}), the rather general approach of \cite{ao} applies also to our scenario. Namely, this essentially consists of finding some bounds for $\# \text{red}_p(\scal(d_K)^{\text{CM}})$.

We say that $d_K$ is  {\em $p$-fundamental} if $p$ does not divide its conductor $c$. Consider $d_{p.f}:= d_K/c^2 \cdot c_p$, where $c_p$ denote the prime-to-$p$ part of $c$. 

\begin{lem}\label{fund}
	A special divisor is equal to the reduction of $\scal(d)^\text{CM}$ for a $p$-fundamental $d$. In particular, we have that $\text{red}_p(\scal(d_K)^{\text{CM}})=\text{red}_p(\scal(d_{p.f})^{\text{CM}})$.
\end{lem}
\begin{proof}
	The spectrum of the completed local ring at any $x\in\cals^{\text{ss}}_{U,\overline{\F}_p}$, i.e., $\Spec\widehat{\calo}_{\cals_{U,\overline{\F}_p},x}$, represents $U$-level structures on the universal deformation of the corresponding supersingular false elliptic curve over $\overline{\F}_p$ to Artin local $\overline{\F}_p$-algebras. By the Serre--Tate theorem for QM abelian surfaces and its corollary \cite[Corollary 4.6]{bz}, this scheme is isomorphic to the scheme representing $U$-level structures on the universal deformation of a supersingular elliptic curve over $\overline{\F}_p$ to Artin local $\overline{\F}_p$-algebras, i.e., the spectrum of the completed local ring of a modular curve. Therefore we deduce, from the classical modular curve case in \cite[(3.3)]{ao}, that also $\text{red}_p(\scal(d_K)^\text{CM})$ and $\text{red}_p(\scal(d_K\cdot p^2)^\text{CM})$ are in each other orbits under the Hecke operator. By the Eichler--Shimura relation (see \cite[Proposition 1.4.10]{sz2} for the Shimura curves case), they are in the same orbit under Frobenius modulo $p$. Hence
	\begin{equation}\label{ES}
		\text{red}_p(\scal(d_K)^\text{CM})=\text{red}_p(\scal(d_K\cdot p^2)^\text{CM}).
	\end{equation}

	By repeating (\ref{ES}), we descend to the desired equality.
\end{proof}

By Lemma \ref{fund}, we henceforth assume that the discriminant is $p$-fundamental. Recall that $d_K=Dc^2$. For $s\in\cals_{U,\overline{\F}_p}^\text{ss}$, let us denote by $\mu_{d_K}(s)$ the counting measure $\#\{x\in\scal(d_K)^\text{CM}: \text{red}_p(x)=s\}$ (which is in parallel with the normalized measure in (\ref{measure})). We immediately have
\begin{equation}\label{trivialbound}
\#\pic(\calo_{D,c})=\sum_{s}\mu_{d_K}(s)\le\#\text{red}_p(\scal(d_K)^\text{CM})\max_s\mu_{d_K}(s),
\end{equation}
where $s\in\cals_{U,\overline{\F}_p}^{\text{ss}}$. Hence we now need to obtain some upper bounds for $\#\{x\in\scal(d_K)^\text{CM}: \text{red}_p(x)=s\}$. By (\ref{trivialbound}) we thus reduce (\ref{lim}) to  showing that
\begin{equation}
	\lim_{p\to\infty}\lim_{d_K\to-\infty}\dfrac{\max_s\mu_{d_K}(s)}{\#\pic(\calo_{D,c})}=0,
\end{equation}
for $p$ not dividing the conductor $c$ of $d_K$. As in Section \ref{ternary}, to a supersingular point $s$ we can attach its Gross lattice $\Lambda_s$ and its corresponding ternary quadratic form $Q_s$. By   Serre--Tate and by Lemma \ref{3.5}, we have 
\[
\mu_{d_{p.f}}\le r^*(Q_s,|d_K|)
\]
where we recall that $ r^*(Q_s,|d_K|)$ denotes the number of primitive representations of $|d_K|$ via $Q_s$. Thus it is enough to show that
\begin{equation}\label{finalbound}
	\lim_{p\to\infty}\lim_{d_{p.f}\to-\infty}\dfrac{r^*(|d_K|,Q_s)}{\#\pic(\calo_{D,c})}=0.
\end{equation}
Such a uniform bound is obtained almost verbatim as in \cite[Sections 5.1,5.3]{ao}: note that, since we deal with non-trivial level structure, the covolume of $\Lambda_s$ needed in the Dirichlet--Hermite bound \cite[Proposition 6.1]{ao} in our case is $2Np$ by Lemma \ref{disc}.

In the case of a superspecial special divisor, the proof follows the very same lines as in the supersingular case, as superspecial implies supersingular, and one can still essentially exploit the analytic bounds of \cite{ao}. A (minor) difference is the cardinality of the superspecial locus: as in (\ref{lim}), we similarly need to let $p$ go to infinity. In fact, by Lemma \ref{numberss} and the explicit formula for the class number of an Eichler order in \cite[p.152]{vign}, we see that 
$\#\cals^{\text{ssp}}_{\F_p}$ grows linearly in $p$.
We also finally point out that for the Gross lattice, one has to consider the construction we described in Sections \ref{Grosslattice} and \ref{ternary} for superspecial points.
\\

In conclusion, let us consider the case of $\mathbb{S}$ containing an infinite subsequence with both the characteristic and the degree going to infinity, i.e., the case (4) of Theorem \ref{hao}.  If the characteristics of the special points do not belong to a finite set, then their closure $Z$ intersects infinitely many fibers, and the divergence of the degrees implies that
\[
\limsup_{p\to \infty} \#(Z \cap \cals_{U,\F_p}) = \infty,
\]
namely, $Z$ has intersection of arbitrarily large order as $p$ goes to infinity. Now we claim that a  closed (hence proper) subscheme $Z$ of $\cals_U$ have bounded intersection with almost all fibers. Therefore we conclude that $Z=\cals_U$. To prove such a claim, it is enough to note that $Z$ does not contain all of $\cals_{U,\Q}=\scal$, hence $Z_\Q$ is finite. By generic flatness, there is some $N>0$ such that $Z$ is flat over $\Z[1/N]$. Therefore we have that $\#(Z\times\Spec\F_p) = \# Z_\Q$ over $\Z[1/N]$. This implies that the size of $Z_{\F_p}$ is bounded for all $p$ not dividing $N$.

%Therefore we have the two following possibilities
%\[
%\overline{\cup\{S : S\in\mathbb{S}\}}^{\text{Zar}}=\begin{cases}
%%	& \cals_\Z;
%	\\
%	& \overline{\{x_1,...,x_n : x_i\in\cals_\Q\}}^{\text{Zar}},
%\end{cases}
%\]
%where the sum of the degrees of the $x_i$ is some fixed $d$. The second case holds a contradiction, since the special fiber $\overline{\cup\{S : S\in\mathbb{S}\}}^{\text{Zar}}\times \Spec\F_p$ has at most $d$ points, so it only gives special divisors of bounded degree, which, by our hypothesis, goes to infinity.


\begin{thebibliography}{100}
	
	\bibitem{akabourbaki}
	Aka, Menny.
	\textit{Joinings classification and applications [after Einsiedler and Lindenstrauss].}
	Ast\'erisque No. 438 (2022), S\'eminaire Bourbaki. Vol. 2021/2022. Expos\'es 1181-1196, Exp. No. 1185, 181-245.
	
	\bibitem{aka}
	Aka, Menny; Luethi, Manuel; Michel, Philippe; Wieser, Andreas. \textit{Simultaneous supersingular reductions of CM elliptic curves.} Journal f\"ur die reine5 und angewandte Mathematik (Crelles Journal), vol. 2022, no. 786, 2022, pp. 1-43.
	
	\bibitem{arth} 
	Arthur, James. \textit{Eisenstein series and the trace formula.} Automorphic forms, representations and $L$-functions (Pro. Sympos. Pure Math., Oregon State Univ., Corvallis, ore., 1977), Part 1, pp. 253-274, Proc. Sumpos. Pure Math, XXXIII, Amer. Math. Soc., Providence, R.I., 1979.
	
	
	\bibitem{bd}
	Bertolini, Massimo; Darmon Henri. \textit{ Heegner points on Mumford-Tate curves.} Invent. Math. 126 (1996), no. 3, 413-456.
	
	\bibitem{brumley}
	Blomer, Valentin; Brumley, Farrell; Khayutin Ilya.
	\textit{The mixing conjecture under GRH.} Arxiv preprint.
	
	\bibitem{bh}
	Blomer, V.; Harcos, Gergely. \textit{Twisted $L$-functions over number fields and Hilbert's eleventh problem.} Geom. Funct. Anal. 20 (2010), no. 1, 1-52.
	
	
	\bibitem{bc} 
	Boutot, Jean-Fran\c{c}ois; Carayol, Henri. \textit{Uniformisation $p$-adique des courbes de Shimura: les th\'eor\'emes de Cerednik et de Drinfeld.} 
	Courbes modulaires et courbes de Shimura (Orsay, 1987-1988). Ast\'erisque No. 196-197 (1991), 7, 45-158 (1992).
	
	\bibitem{bz}
	Boutot J.-F.; Zink Thomas. \textit{The $p$-adic Uniformization of Shimura Curves.} Arxiv preprint.
	
	\bibitem{buzzard}
	Buzzard, Kevin.
	\textit{Integral models of certain Shimura curves.}
	Duke Math. J. 87 (1997), no. 3, 591-612.
	
    \bibitem{car}
    Carayol, H.
     \textit{Sur la mauvaise r\'eduction des courbes de Shimura.}(French) Compositio Math.59(1986), no.2, 151-230.
     
     
     \bibitem{cco}
     Chai, Ching-Li; Conrad, Brian; Oort, Frans.
     \textit{Complex multiplication and lifting problems.}
     Math. Surveys Monogr., 195
     American Mathematical Society, Providence, RI, 2014. 
     
	
	%\bibitem{cor}
	%Cornut, Christophe. \textit{Mazur's conjecture on higher Heegner points.} Invent. Math. 148 (2002), no. 3, 495-523. 
	
	\bibitem{cj}
	Cornut, Christophe; Jetchev, Dimitar.
	\textit{Liftings of reduction maps for quaternion algebras.} Bull. Lond. Math. Soc.45(2013), no.2, 370-386.
	
	\bibitem{cv}
	Cornut, C.; Vatsal, Vinayak. \textit{CM points and quaternion algebras.} Doc. Math. 10 (2005), 263-309.
	
	%\bibitem{cv2}
	%Cornut, C.; Vatsal, V. \textit{Nontriviality of Rankin-Selberg $L$-functions and CM points.} $L$-functions and Galois representations, 121-186, London Math. Soc. Lecture Note Ser., 320, Cambridge Univ. Press, Cambridge, 2007.
	
	\bibitem{cox}
	Cox, David A. \textit{Primes of the form $x^2+y^2=1$.} Fermat, class field theory and complex multiplication
	Wiley-Intersci. Publ.
	John Wiley \& Sons, Inc., New York, 1989. xiv+351 pp.
	
	
	\bibitem{del}
	Deligne, Pierre. \textit{Travaux de Shimura.} S\'eminaire Bourbaki, Exp. No. 389, pp. 123-165. Lecture Notes in Math., Vol. 244, Springer, Berlin, 1971.
	
	
	
	\bibitem{dis2}
	Disegni, Daniel. \textit{The $p$-adic Gross--Zagier formula on Shimura curves.} Compositio Mathematica. 2017;153(10):1987-2074. 
	
	\bibitem{dis}
	Disegni, D.
	\textit{$p$-adic equidistribution of CM points}.  Comment. Math. Helv. 97 (2022), no. 4, 635-668. 
	

	
	\bibitem{drin}
	Drinfeld, Vladimir G.
	\textit{Coverings of p-adic symmetric domains.} (Russian) Funkcional. Anal. i Prilozhen. 10 (1976), no. 2, 29-40. 
	
	\bibitem{edix}
	Edixhoven, Bas.
	\textit{On the Andr\'e-Oort conjecture for Hilbert modular surfaces.} Moduli of abelian varieties (Texel Island, 1999), 133-155.
	Progr. Math., 195
	Birkh\"auser Verlag, Basel, 2001.
	
	\bibitem{elkies}%for intro
	Elkies, Noam; Ono, Ken; Yang, Tonghai.
	\textit{Reduction of CM elliptic curves and modular function congruences.} Int. Math. Res. Not.(2005), no.44, 2695-2707.
	
	\bibitem{fli}
	Flicker, Yuval Z. \textit{Automorphic forms on covering groups of $\GL(2)$.} Invent. Math. 57
	(1980), no. 2, 119-182.
	
	\bibitem{fried}
	Friedman, Eduardo.
	\textit{Analytic formulas for the regulator of a number field.} Invent. Math.98(1989), no.3, 599-622.
	
	\bibitem{gar}
	Garrett, Paul B. 
	\textit{Holomorphic Hilbert modular forms.} The Wadsworth \& Brooks/Cole Mathematics Series. Wadsworth \& Brooks/Cole Advanced Books \& Software, Pacific Grove, CA, 1990. xvi+304 pp.
	
	\bibitem{gar}
	Garrett, Paul.
	\textit{Modern analysis of automorphic forms by example.} Vol. 1.
	Cambridge Stud. Adv. Math., 173
	Cambridge University Press, Cambridge, 2018. 
	
	\bibitem{gel}
	Gelbart Stephen. 
	\textit{Weil's Representation and the Spectrum of the Metaplectic Group.} Lecture Notes in Math., Vol. 530
	Springer-Verlag, Berlin-New York, 1976.
	
	\bibitem{gps1}
	Gelbart, S.; Piatetski-Shapiro, Ilya. \textit{Distinguished Representations and Modular Forms of Half-integral Weight} Invent. Math. 59 (1980), no. 2, 145-188.
	
	\bibitem{gps}
	Gelbart, S., Piatetski-Shapiro, I. \textit{On Shimura's correspondence for modular forms of half-integral weight.} Proc. Int. Coll. Auto. Forms, Rep. Theory and Arith., Bombay 1980.
	
	\bibitem{gr}
	Gross, Benedict. \textit{Heights and the special values of $L$-series.} Number Theory (Montreal, Que., 1985). In: CMS Conference Proceedings, vol. 7, pp. 115-187. American Mathematical Society, Providence, 1987.
	
	\bibitem{hanke}
	Hanke, Jonathan. \textit{Local densities and  explicit bounds for representability by a quadratic form.}  Duke Math. J. 124(2): 351-388, 2004.
	
	\bibitem{hanke2}
	Hanke, J. \textit{Quadratic Forms and Automorphic Forms.} Dev. Math., 31
	Springer, New York, 2013, 109-168.
	
	\bibitem{kohnen}
	Hiraga Kaoru, Ikeda Tamotsu. \textit{On the Kohnen plus space for Hilbert modular forms of half-integral weight I.} Compositio Mathematica. 2013;149(12):1963-2010.
	
	\bibitem{jk}
	Jetchev, Dimitar; Kane, Ben. \textit{Equidistribution of Heegner points and ternary quadratic forms.} Math. Ann. 350 (2011), no. 3, 501-532.
	
	%\bibitem{jon}
	%Jones, Burton. \textit{The arithmetic theory of quadratic forms.} Carcus Monograph Series, no. 10. The Mathematical Association of America, (Buffalo) 1950.
	
	\bibitem{kass}
	Kassaei, Payman L.; \textit{$\mathcal{P}$-adic modular forms over Shimura curves over totally real fields.} Compositio Mathematica.  140 (2) (2004), pp. 359-395. 
	
	
	
	%\bibitem{katz}
	%Katz, Nicholas M.;\textit{ Serre--Tate local moduli.} In: Giraud, J., Illusie, L., Raynaud, M. (eds) Surfaces Algébriques. Lecture Notes in Mathematics, vol 868. Springer, 1981.
	
	
	\bibitem{KM}
	Katz, N.M.; Mazur, Barry.
	\textit{Arithmetic moduli of elliptic curves.}
	Ann. of Math. Stud., 108
	Princeton University Press, Princeton, NJ, 1985.
	
	\bibitem{kha}
	Khayutin, Ilya.
	\textit{Joint equidistribution of CM points.}
	Ann. of Math. (2) 189 (2019), no. 1, 145-276.
	
	 
	\bibitem{ks}
	Kim, Henry H.; Shahidi, Freydoon.
	\textit{Cuspidality of symmetric powers with applications.} Duke Math. J.112(2002), no.1, 177-197.
	
	\bibitem{rotger}
	Kontogeorgis, Aristides; Rotger, Victor.
	\textit{On the non-existence of exceptional automorphisms on Shimura curves.}
	Bull. Lond. Math. Soc. 40 (2008), no. 3, 363-374.
	
	  \bibitem{kudla}
	Kudla, Stephen S.; Rallis, Stephen.
	\textit{On the Weil-Siegel formula.}
	J. Reine Angew. Math. 387 (1988), 1-68.
	
	\bibitem{kudla2}
	Kudla, S. S.; Rallis, S.
	\textit{On the Weil-Siegel formula. II. The isotropic convergent case.}
	J. Reine Angew. Math. 391 (1988), 65-84.
	
	
	\bibitem{lab}
	Labesse, Jean-Pierre \textit{The Langlands spectral decomposition.} London Math. Soc. Lecture Note Ser., 467
	Cambridge University Press, Cambridge, 2021, 176-214.
	
	\bibitem{lang}
	Lang, Serge.
	\textit{Algebraic number theory.}
	Second edition,
	Grad. Texts in Math., 110
	Springer-Verlag, New York, 1994. 
	
	
	\bibitem{michel}
	Michel, Philippe. \textit{The subconvexity problem for Rankin-Selberg  $L$-functions and equidistribution of Heegner points.} Ann. of Math. (2) 160 (2004), no. 1, 185-236. 
	
	\bibitem{wn}
	Narkiewicz, W. \textit{Elementary and Analytic Theory of Algebraic Numbers.}  (2nd edition), Springer, 1990.
	
	 \bibitem{ogg}
	Ogg, Andrew P.
	\textit{Real points on Shimura curves.} Arithmetic and geometry, Vol. I, 277-307.
	Progr. Math., 35
	Birkhäuser Boston, Inc., Boston, MA, 1983. 
	
	\bibitem{rag}
	Raghunatan A., Tanabe Naomi. \textit{Notes on the arithmetic of Hilbert modular forms.} J. Ramanujan Math. Soc.26(2011), no.3, 261319.
	
	\bibitem{aoexp}
	Richard, Rodolphe.
	\textit{Problemes de type Andr\'e-Oort en pinceau arithm\'etique.}
    Exp. Math.,vol. 41,2023, pp. 618-630.
	
	\bibitem{ao}
	Richard, R.
	\textit{A two-dimensional arithmetic Andr\'e-Oort problem.} Bull. Lond. Math. Soc.55(2023), no.3, 1459-1488.
	
	
	
	\bibitem{fms}
     Saettone, Francesco Maria. \textit{Equidistribution of CM points on a Shimura curve modulo a ramified prime.}
     Arxiv preprint.
	
	\bibitem{sp}
	Schulze-Pillot, Rainer. \textit{Thetareihen positiv definiter quadratischer Formen.} Invent. Math. 75 (1984),
	283-299. 
	
	\bibitem{shi}
     Shimura, Goro. \textit{On Hilbert modular forms of half-integral weight.}  Duke Math. J. 55(4), 765-838 (1987).
	
	\bibitem{vat}
	Vatsal, Vinayak. \textit{Uniform distribution of Heegner points.} Invent. Math. 148 (2002), no. 1, 1-46.
	
	
	\bibitem{vign}
	Vign\'eras, Marie-France. \textit{Arithm\'etique des alg\`ebres de quaternions.} Lecture Notes in Mathematics, 800. Springer, Berlin, 1980.
	
	
	
	\bibitem{voi}
	Voight, John. \textit{ Quaternion algebras.} Graduate Texts in Mathematics, 288. Springer, 2021.
	
	
	\bibitem{yzz}
	Yuan, Xinyi; Zhang, Shou-Wu; Zhang, Wei. \textit{The Gross-Zagier formula on Shimura curves.} Annals of Mathematics Studies, 184. Princeton University Press, Princeton, NJ, 2013. x+256 pp. 
	

	
	\bibitem{sz}
	Zhang, Shou-Wu.
	\textit{Gross-Zagier formula for $\GL_2$}.
	Asian J. Math. 5 (2001). no. 2, 183-290.
	
	
	\bibitem{sz2}
	Zhang, Shou-Wu.
	\textit{ Heights of Heegner points on Shimura curves}. Ann. of Math. (2) 153 (2001), no. 1, 27-147.
	
	\bibitem{swzequi}
	Zhang, Shou-Wu.
	\textit{Equidistribution of CM-points on quaternion Shimura varieties.}
	Int. Math. Res. Not. 2005, no. 59, 3657-3689.
	
	\bibitem{xz}
	Zhang, Xiaoyu.
	\textit{Simultaneous supersingular reductions of Hecke orbits.} Arxiv preprint.
	
	
	
	
\end{thebibliography}
\end{document}